\documentclass[12pt]{amsart}
\usepackage{amsmath}
\usepackage{amssymb}
\usepackage{latexsym}

\newcommand{\halmos}{\rule{5pt}{5pt}}

\numberwithin{equation}{section}
\newtheorem{thm}{Theorem}[section]
\newtheorem{prop}[thm]{Proposition}
\newtheorem{lemma}[thm]{\bf Lemma}

\newtheorem{cor}[thm]{\bf Corollary}

\theoremstyle{definition} 

\theoremstyle{remark}

\begin{document}
\title[Special solutions of $q$-Heun equation by $q$-integral transformations]
{Special solutions of $q$-Heun equation by $q$-integral transformations}
\author{Ayaka Murakami}
\address{Department of Mathematics, Ochanomizu University, 2-1-1 Otsuka, Bunkyo-ku, Tokyo 112-8610, Japan}
\email{ayaka.murakami.math@gmail.com}
\author{Kouichi Takemura}
\address{Department of Mathematics, Ochanomizu University, 2-1-1 Otsuka, Bunkyo-ku, Tokyo 112-8610, Japan}
\email{takemura.kouichi@ocha.ac.jp}
\begin{abstract}
We obtain special solutions of the $q$-Heun equation which are expressed as finite summations of $q$-hypergeometric functions.
These solutions are obtained by considering the $q$-integral transformations of the polynomial-type solutions.
\end{abstract}
\maketitle

\section{Introduction}

The $q$-Heun equation
\begin{align}
& \{a_{2}x^{2} + a_{1}x + a_{0}\}g(x/q) - \{b_{2}x^{2} + b_{1}x + b_{0} \}g(x) 
 \label{eq:qHeun0} \\
& + \{c_{2}x^{2} + c_{1}x + c_{0} \}g(xq) = 0, \qquad a_{2}a_{0}c_{2}c_{0} \neq 0 \notag
\end{align}
was introduced by Hahn \cite{Hahn} in 1971.
As $q \to 1$, the $q$-Heun equation tends to Heun's differential equation (see \cite{Hahn,TakR}).
Although special cases of equation (\ref{eq:qHeun0}) had appeared in some papers (e.g.~\cite{MP,WZ}), it seems that the discovery of Hahn had been forgotten for a long time because the paper \cite{Hahn} had not been cited before 2016. 
The $q$-Heun equation was rediscovered by the second author \cite{TakR} through study of integrable systems, namely of degenerations of Ruijsenaars-van Diejen system \cite{vD0,RuiN}.
Let $T_{x}^{\pm 1}$ be the $q$-shift operator defined by $T_{x}^{\pm 1} g(x) = g(q^{\pm 1}x)$ and 
\begin{align}
  &A^{\langle 4 \rangle} (x; h_{1}, h_{2}, l_{1}, l_{2}, \alpha_{1}, \alpha_{2}, \beta) \label{eq:A4} \\
  &\; = x^{-1} (x-q^{h_{1} + 1/2} t_{1}) (x-q^{h_{2} +1/2} t_{2}) T_{x}^{-1}  + q^{\alpha_{1} + \alpha_{2}} x^{-1} (x - q^{l_{1} -1/2} t_{1})(x - q^{l_{2} - 1/2} t_{2}) T_{x} \notag \\
  &\qquad \quad - \{(q^{\alpha_{1}} + q^{\alpha_{2}})x + q^{(h_{1} + h_{2} + l_{1} + l_{2} + \alpha_{1} + \alpha_{2})/2}(q^{\beta/2} + q^{-\beta/2}) t_{1} t_{2} x^{-1}\} \notag
\end{align}
be the operator which is obtained by degenerating the one variable Ruijsenaars-van Diejen operator four times.
Then, the $q$-Heun equation admits the expression as
  \begin{align}
    A^{\langle 4 \rangle} (x; h_{1}, h_{2}, l_{1}, l_{2}, \alpha_{1}, \alpha_{2}, \beta) g(x) = E g(x), \quad \text{$E \in \mathbb{C}$}. \label{eq:q-Heun}
  \end{align}
Here, the parameter $b_{1}$ in equation (\ref{eq:qHeun0}) essentially corresponds to the eigenvalue $E$.
We regard the parameter $b_{1}$ or the eigenvalue $E$ as an accessory parameter, which is a $q$-analogue of the accessory parameter of Heun's differential equation.

The Heun equation or Heun's differential equation is a standard form of the second order linear differential equation with four regular singularities on the Riemann sphere, and it is written as 
\begin{align}
  \frac{{\rm d}^{2}y}{{\rm d}z^{2}} + \left(\frac{\gamma}{z} + \frac{\delta}{z - 1} + \frac{\epsilon}{z - t} \right) \frac{{\rm d}y}{{\rm d}z} + \frac{\alpha \beta z - B}{z(z - 1)(z - t)}y = 0, \label{eq:Heun}
\end{align}
under the condition $\gamma + \delta + \epsilon = \alpha + \beta + 1$.
The parameter $B$ in equation (\ref{eq:Heun}) is called the accessory parameter, and it depends on the global structure of solutions.
Although it is difficult to calculate the global structure of solutions for each $B \in \mathbb{C}$, some special solutions are known for special values of the accessory parameter $B$ and the other parameters.
In particular, there exist special solutions to equation (\ref{eq:Heun}) which are expressed as finite summation of the hypergeometric functions (see \cite{TakIT}), and they are related to integral transformations of solutions of Heun's differential equation.
Kazakov and Slavyanov \cite{KS} essentially obtained the integral transformation as follows.
Let $v(w)$ be a solution of Heun's differential equation in the form
\begin{align*}
  \frac{{\rm d}^2 v}{{\rm d}w^2} + \left( \frac{\gamma '}{w}+\frac{\delta '}{w-1}+\frac{\epsilon '}{w-t}\right) \frac{{\rm d}v}{{\rm d}w} + \frac{\alpha ' \beta ' w -B '}{w(w - 1)(w - t)} v = 0 ,
\end{align*}
and let $C$ be a suitable integral path in the $w$-plane.
Then, the function
\begin{align*}
  y(z)=\int _{C} v(w) (z-w)^{-\alpha } {\rm d}w
\end{align*}
is a solution to equation (\ref{eq:Heun}), if the parameters satisfy
\begin{align*}
  &\gamma '=\gamma +1 -\alpha , \qquad \delta' =\delta +1-\alpha , \qquad \epsilon '=\epsilon +1-\alpha , \qquad \alpha '=2-\alpha , \\
  &\beta '= -\alpha +\beta +1 ,\qquad B'=B+(1-\alpha )(\epsilon +\delta t+(\gamma -\alpha ) (t+1)).
\end{align*}
Note that the integral transformation of the Heun equation was also derived by the middle convolution \cite{TakI,TakM} where the integral path was taken as the Pochhammer contours, and the integral transformation was applied to obtain the special solutions of the Heun equation which are expressed as finite summations of the hypergeometric equation \cite{TakIT}.

A formal $q$-integral transformation of the $q$-Heun equation was found in \cite{STT2} through the study of the $q$-middle convolution on the linear $q$-difference equations which are related to the $q$-Painlev\'e VI.
A rigorous form of the $q$-integral transformation of the $q$-Heun equation was formulated by developing a theory of the kernel function identity in \cite{TakK}, which we describe in Theorem \ref{thm:q-trans}.

It is known that the degenerated Ruijsenaars-van Diejen operator in equation (\ref{eq:A4}) preserves a finite-dimensional space if the parameters satisfies a relation \cite{TakqH}, and it is related to the polynomial-type solutions of the $q$-Heun equations \cite{KST}.
On the other hand, the $q$-Heun equation may have a hypergeometric solution, if the parameters are special.
The ${}_r\phi_{r-1}$ basic hypergeometric series is defined by
  \begin{align}
    {}_r\phi_{r-1} \biggl( \begin{array}{@{}c@{}} a_{1}, \cdots, a_{r} \\
b_{1}, \cdots, b_{r-1} \end{array} ;q,z \biggr) = \sum_{n = 0}^{\infty} \frac{(a_{1}; q)_{n} \cdots (a_{r}; q)_{n}}{(q; q)_{n}(b_{1}; q)_{n} \cdots (b_{r-1}; q)_{n}} z^{n}. \label{eq:bhs}
  \end{align}
It is well known that the basic hypergeometric function $_2 \phi _1 (a ,b ;c ;x) $ satisfies the basic (or the $q$-difference) hypergeometric equation 
\begin{equation*}
(x-q) f(x/q) - ((a+b)x -q-c)f(x)+ (abx-c)f(q x)=0, 
\end{equation*}
whose expression is simpler than the $q$-Heun equation (\ref{eq:qHeun0}).
The variant of the $q$-hypergeometric equation of degree two was introduced in \cite{HMST} as
\begin{align}
& (x-q^{h_1 +1/2} t_1) (x - q^{h_2 +1/2} t_2) g(x/q) + q^{\alpha _1 +\alpha _2} (x - q^{l_1-1/2}t_1 ) (x - q^{l_2 -1/2} t_2) g(q x) \label{eq:varqhgdeg2} \\
&  -[ (q^{\alpha _1} +q^{\alpha _2} ) x^2 +E x + p ( q^{1/2}+ q^{-1/2}) t_1 t_2 ] g(x) =0, \nonumber \\
& p= q^{(h_1 +h_2 + l_1 + l_2 +\alpha _1 +\alpha _2 )/2 } , \quad E= -p \{ (q^{- h_2 }+q^{-l_2 })t_1 + (q^{- h_1 }+ q^{- l_1 }) t_2 \} \nonumber .
\end{align}
This equation is obtained by restricting the parameters of the $q$-Heun equation in the form of equation (\ref{eq:q-Heun}). 
It was established in \cite[Theorem 4.5]{AT} that the following functions
 \begin{align}
& h_1 (x) = x^{ \lambda _{1} } \frac{( q^{ \lambda _{1} - h_{1} + \alpha _{1} + 1/2}x/t_{1}; q)_\infty}{( q^{-h_{1} + 1/2}x/t_{1}; q)_\infty} \label{eq:solvarHG2} \\
  &\quad \times {}_3\phi_2 \biggl( \begin{array}{@{}c@{}} q^{ \lambda _{1} -h_{1} + l_{1} + \alpha _{1} }, q^{ \lambda _{1} -h_{1} + l_{2} + \alpha _{1} }t_{2}/t_{1}, q^{-h_{1} + 1/2}x/t_{1} \\
     q^{- h_{1} + h_{2} + 1}t_{2}/t_{1}, q^{ \lambda _{1} - h_{1} + \alpha _{1} + 1/2}x/t_{1} \end{array} ;q,q^{ \lambda _{1}  + \alpha_{2} } \biggr), \notag \\
  & h_2 (x) = x^{ \lambda _{1} } \frac{( q^{ \lambda _{1} - h_{2} + \alpha _{1} + 1/2}x/t_{2}; q)_\infty}{( q^{-h_{2} + 1/2}x/t_{2}; q)_\infty} \notag \\
  &\quad \times {}_3\phi_2 \biggl( \begin{array}{@{}c@{}} q^{ \lambda _{1} -h_{2} + l_{1} + \alpha _{1} }t_{1}/t_{2}, q^{ \lambda _{1} -h_{2} + l_{2} + \alpha _{1} }, q^{-h_{2} + 1/2}x/t_{2} \\
     q^{h_{1} - h_{2} + 1}t_{1}/t_{2}, q^{ \lambda _{1} - h_{2} + \alpha _{1} + 1/2}x/t_{2} \end{array} ;q,q^{ \lambda _{1}  + \alpha_{2} } \biggr), \notag \\
 & h_3 (x) =x^{ - \alpha_{2} } \frac{(q^{- \lambda _{1} + h_{1} - \alpha _{1} + 3/2}t_{1}/x, q^{- \lambda _{1} + h_{2} - \alpha _{1} + 3/2}t_{2}/x ; q)_{\infty}}{(q^{l_{1} + 1/2}t_{1}/x, q^{l_{2} + 1/2}t_{2}/x ; q)_{\infty}} \notag \\
  &\quad \times {}_3\phi_2 \biggl( \begin{array}{@{}c@{}} q^{l_{1} + 1/2}t_{1}/x, q^{l_{2} + 1/2}t_{2}/x, q^{- \lambda _{1} - \alpha _{1} + 1} \\
     q^{- \lambda _{1} + h_{1} - \alpha _{1} + 3/2}t_{1}/x, q^{- \lambda _{1} + h_{2} - \alpha _{1} + 3/2}t_{2}/x \end{array} ;q,q^{ \lambda _{1}  + \alpha_{2} } \biggr). \notag 
  \end{align}
are solutions of equation (\ref{eq:varqhgdeg2}), if $\lambda _1 +\alpha _2 >1 $.
We can regard these functions as the special solutions to the $q$-Heun equation.

In this paper, we obtain special solutions of the $q$-Heun equation which are written as finite summations of the $q$-hypergeometric functions by using the $q$-integral transformation of the $q$-Heun equation.
Our special solutions consist of two families.
One is a generalization of the solutions in equation (\ref{eq:solvarHG2}), that is described in Theorem \ref{thm:2sols}.
The parameters of the $q$-Heun equation for the solutions satisfy the condition that $\beta $ is a positive integer and the singularity $x=0$ of the $q$-difference equation is apparent, whose condition is described as a relation for the accessory parameter $E$.
See Section \ref{sec:appsing} for the apparent singularity.
Another family of our special solutions is written as the finite summations of $_2 \phi _1 $, and the parameters satisfies the condition that $ l_{2} - h_{2} $ is a positive integer and  the accessory parameter $E$ satisfies a certain relation.
See Theorem \ref{thm:1sols} for details of the solutions.

This paper is organized as follows.
In Section \ref{sec:prelqint}, we review the $q$-integral transformations and preliminary things which are related to the $q$-Heun equation.
We describe the main results in Section \ref{sec:q-Int}.
In Section \ref{sec:pow}, we review power series solutions of the $q$-Heun equation and the polynomial-type solutions.
In Section \ref{sec:appsing}, some properties on an apparent singularity about $x=0$ are explained.
In Section \ref{sec:1sol}, we discuss the first family of the solutions of the $q$-Heun equation.
In Section \ref{sec:2sol}, we discuss the second family of the solutions of the $q$-Heun equation.
In Section \ref{sec:rem}, we give concluding remarks.

Throughout this paper, we assume that the complex number $q$ satisfies $0 <|q|<1$.

\section{$q$-integral transformation and gauge-transformation} \label{sec:prelqint}

\subsection{Preliminaries}
The $q$-shifted factorials are defined by
  \begin{align*}
    &(a; q)_{n} :=
    \begin{cases}
      1, & n = 0, \\
      (1 - a) (1-aq) \cdots(1 - aq^{n - 1}), & n = 1, 2, \cdots , \\
     1/(aq^{n}; q)_{-n}, & n = -1, -2, \cdots ,
    \end{cases} \\
    &(a; q)_{\infty} := \prod_{k = 0}^{\infty} (1 - aq^{k}) .
  \end{align*}
We have
  \begin{align}
 & (a; q)_{n} = \frac{(a; q)_{\infty}}{(aq^{n}; q)_\infty}, \quad n \in \mathbb{Z} , \label{eq:GR1.2.30} \\
 & \frac{1}{(q; q)_{k}} = 0, \qquad \qquad  k \in \mathbb{Z}_{< 0} . \label{eq:qshift0}
  \end{align}
We use the notation such that
  \begin{align*}
    (a_{1}, \cdots, a_{m}; q)_{n} = (a_{1}; q)_{n} \cdots (a_{m}; q)_{n}, \quad (n \in \mathbb{Z} \cup \{\infty\}).
  \end{align*}
We adopt the definition of the theta function as
\begin{align} 
\label{def:theta}
    \vartheta_{q}(t) := (t, q/t, q; q)_{\infty}.
\end{align}
Then, we have 
  \begin{lemma} \label{lem:theta}
If $K$ is an integer, then 
   \begin{align*}
      \frac{\vartheta_{q}\left(q^{K}t \right)}{\vartheta_{q}\left(q^{K}s \right)} = \left(\frac{s}{t}\right)^{K} \frac{\vartheta_{q}\left(t \right)}{\vartheta_{q}\left(s \right)}.
    \end{align*}
  \end{lemma}
The summation 
  \begin{align*}
    \int_{0}^{\xi \infty} f(s) {\rm{d}}_{q}s = (1 - q) \sum_{n = - \infty}^{\infty} q^{n} \xi f(q^{n}   \xi)
  \end{align*}
is called the Jackson integral, where $\xi \neq 0 $.
The usual integration on $(0, +\infty )$ is obtained by the limit $q \to 1$.

\subsection{$q$-integral transformations of $q$-Heun equations}
We recall $q$-integral transformations which are related to the $q$-Heun equations.
\begin{thm} [\cite{TakK}, $q$-integral transformations of $q$-Heun equations] \label{thm:q-trans}
Assume that the function $h(s)$ satisfies
\begin{align*}
& A^{\langle 4 \rangle} ( s; h'_1, h'_2, l'_1, l'_2, \alpha '_1, \alpha '_2, \beta ' ) h(s) = E' h(s) 
\end{align*}
and
\begin{align}
& \lim _{L \to +\infty } \frac{h(s)}{s^{(h'_1 + h'_2 - l'_1 - l'_2 - \alpha '_1 - \alpha '_2 + \beta ' +2)/2}} \big| _{s=q^{L} \xi } = C_1, \;  \lim _{K \to -\infty } \frac{h(s)}{s^{-\alpha ' _1}} \big| _{s=q^{K} \xi } = C_2 \label{eq:A4hslim}
\end{align}
for some constants $E', C_1 ,C_2$, and the variable $\xi $ is independent of the variable $x$ or it is proportional to $x$ (i.~e.~$\xi =Ax$ where $A $ is independent of $x$).
Let $h_1, h_2, l_1, l_2, \alpha _1, \alpha _2$, $\beta , E , \chi ,  \mu , \mu _0 $ be the parameters which satisfy
\begin{align}
    &  \chi = ( l'_1 + l'_2 - h'_1 - h'_2 - \alpha '_1 + \alpha '_2 - \beta ' )/2 , \quad \mu = \mu_{0} + 1 + \chi,  \label{eq:chi'} \\
    &E = q^{\mu _0 + \alpha _1 - \alpha ' _1 } E', \quad \beta = - \beta ' - \chi , \quad \alpha _2 = \alpha _1 - \alpha ' _1 + \alpha ' _2 - \chi , \nonumber \\
    & l_i = l'_i +\mu _0 , \quad h_i = h'_i +\mu _0 +\chi , \quad i = 1, 2 . \nonumber 
\end{align}
Set
\begin{align*}
 P_{\mu, \mu_{0}}^{(1)} (x, s) := \frac{(q^{\mu} s/x; q)_{\infty}}{(q^{\mu_{0}} s/x; q)_{\infty}}, \;  P_{\mu, \mu_{0}}^{(2)} (x, s) := \left(\frac{x}{s} \right)^{\mu - \mu_{0}} \frac{(q^{-\mu_{0} + 1} x/s; q)_{\infty}}{(q^{-\mu + 1} x/s; q)_{\infty}} .
\end{align*}
(i) The Jackson integral
\begin{align}
& g (x) = x^{- \alpha _1}  \int^{\xi \infty }_{0} s^{-(h'_1 + h'_2 - l'_1 - l'_2 - \alpha '_1 - \alpha '_2 + \beta ' +2)/2} h(s) P^{(1)}_{\mu ,\mu _0 } (x,s)  \, d_{q}s \label{eq:A4Jint}
\end{align}
converges and it satisfies
\begin{align}
& A^{\langle 4 \rangle} ( x;  h_1, h_2, l_1, l_2, \alpha _1, \alpha _2, \beta  ) g (x) =  E g(x) + (1-q) (k_2 (x) - k_1 (x)), \label{eq:thmA4gE1}
\end{align}
where
\begin{align*}
& k_1 (x)= C_1 x^{- \alpha _1} q^{\mu _0 + \alpha _1 + h'_1 + h'_2 + \chi } ( q^{ \beta ' } - 1 ) t_1 t_2 ,  \\
& k_2(x)= C_2 x^{- \alpha _1 } \frac{\vartheta _q (q^{- \mu _0 - \chi } x/\xi  )}{\vartheta _q (q^{- \mu _0 +1 } x/\xi  )} \xi ^{\chi + 1 } q^{\mu _0 + \alpha _1 } ( q^{\alpha ' _2 - \alpha '_1}  -  1 ) . \nonumber
\end{align*}
(ii)  The Jackson integral
\begin{align}
& g(x) = x^{-\alpha_{1}} \int_{0}^{\xi \infty} s^{-(h_{1}' + h_{2}' - l_{1}' - l_{2}' - \alpha_{1}' - \alpha_{2}' + \beta' + 2)/2} h(s) P_{\mu, \mu_{0}}^{(2)} (x, s) {\rm{d}}_{q}s \label{eq:A4Jint2}
\end{align}
converges and it satisfies
\begin{align}
& A^{\langle 4 \rangle} ( x;  h_1, h_2, l_1, l_2, \alpha _1, \alpha _2, \beta  ) g (x) =  E g(x) + (1-q) (k_2 (x) - k_1 (x)), \label{eq:thmA4gE2}
\end{align}
where 
\begin{align*}
& k_{1}(x) = C_{1} x^{- \alpha _1+ \chi +1} \xi ^{ -\chi -1 } \frac{\vartheta_{q} \big(q^{- \mu _0 +1 } x/\xi \big)}{\vartheta_{q} (q^{- \mu _0 -\chi } x/\xi )} q^{\mu_{0 }+ \alpha_{1} + h'_{1} + h'_{2} + \chi } \big( q^{ \beta' } - 1 \big) t_{1} t_{2}, \\
& k_{2}(x) = C_{2} x^{- \alpha _1 + \chi +1} q^{\mu_{0} + \alpha _1 } \big( q^{\alpha'_{2} - \alpha'_{1}} - 1 \big)
 . \nonumber 
\end{align*}
\end{thm}
Note that (\ref{eq:chi'}) is equivalent to 
\begin{align}
    &\chi = (h_{1} + h_{2} - l_{1} - l_{2} + \alpha_{1} - \alpha_{2} - \beta )/2, \quad \mu = \mu_{0} + 1 + \chi,  \label{eq:chi} \\
    &E' = q^{- \mu_{0} + \alpha '_{1} - \alpha_{1}} E, \quad \beta ' = -\beta - \chi, \quad \alpha '_{2} = \alpha ' _{1} - \alpha_{1} + \alpha_{2} + \chi, \nonumber \\
    &l'_{i} = l_{i} - \mu_{0}, \quad h'_{i} = h_{i} - \mu_{0} - \chi, \quad i = 1, 2 \nonumber 
 \end{align}
In the theorem, if $C_1=C_2=0$, then equations (\ref{eq:thmA4gE1}) and (\ref{eq:thmA4gE2}) are written as
\begin{align*}
& A^{\langle 4 \rangle} ( x;  h_1, h_2, l_1, l_2, \alpha _1, \alpha _2, \beta  ) g (x) =  E g(x) ,
\end{align*}
and we obtain the $q$-integral transformations of solutions to the $q$-Heun equations with different parameters.

By applying the $q$-integral transformation in Theorem \ref{thm:q-trans}, some solutions of the $q$-hypergeometric equation and of the non-homogeneous version of the variant of the $q$-hypergeometric equation of degree $2$ were rederived in \cite{TakK}, and these solutions are regarded as special solutions to the $q$-Heun equation.
In this paper, we obtain more solutions to the $q$-Heun equation by using Theorem \ref{thm:q-trans} and the gauge transformation in the next subsection. 

\subsection{Gauge transformation of $q$-Heun equation}
We recall a proposition on the gauge transformation of the $q$-Heun equation, which will be used later.
\begin{prop} [{cf.\ \cite[Proposition 5.2]{TakK}}] \label{prop:GT}
Let $(i,j) = (1,2)$ or $(2,1)$. 
If the function $f(x) $ satisfies
\begin{align}
& x^{-1} (x-q^{l_i + 1/2} t_i) (x-q^{h_j +1/2} t_j) f(x/q) +  q^{\alpha _1 +\alpha _2} x^{-1} (x - q^{h_i -1/2} t_i) (x - q^{l_j -1/2} t_j) f(qx)  \nonumber \\
& -\{ (q^{\alpha _1} +q^{\alpha _2} ) x  + q^{(h_1 +h_2 + l_1 +l_2 +\alpha _1 +\alpha _2 )/2} ( q^{\beta/2} + q^{-\beta/2} ) t_1 t_2 x^{-1} \} f(x) = E f(x)  \nonumber 
\end{align}
and
\begin{align}
& g_1(x) = \frac{(q^{h_i + 1/2} t_i/x; q)_{\infty }}{(q^{l_i + 1/2} t_i/x; q)_{\infty }} f(x) , \; g_2 (x) =  x ^{h_i -l_i} \frac{(x/(q^{l_i - 1/2} t_i); q)_{\infty }}{(x/(q^{h_i - 1/2} t_i) ; q)_{\infty }} f(x),  \nonumber 
\end{align}
then the functions $g(x)= g_1 (x)$ and $g(x)= g_2 (x)$ satisfy
\begin{align}
& x^{-1} (x-q^{h_1 + 1/2} t_1) (x-q^{h_2 +1/2} t_2) g(x/q) +  q^{\alpha _1 +\alpha _2} x^{-1} (x - q^{l_1 -1/2} t_1) (x - q^{l_2 -1/2} t_2) g(qx)  \nonumber \\
& -\{ (q^{\alpha _1} +q^{\alpha _2} ) x  + q^{(h_1 +h_2 + l_1 +l_2 +\alpha _1 +\alpha _2 )/2} ( q^{\beta/2} + q^{-\beta/2} ) t_1 t_2 x^{-1} \} g(x) = E g(x) . \nonumber 
\end{align}
\end{prop}

\section{Special solutions of $q$-Heun equation} \label{sec:q-Int}
\subsection{Power series solutions of $q$-Heun equation and related solutions} \label{sec:pow}

Let us consider solutions to the $q$-Heun equation
  \begin{align}
    A^{\langle 4 \rangle} (x; h''_{1}, h''_{2}, l''_{1}, l''_{2}, \alpha ''_{1}, \alpha ''_{2}, \beta '') g(x) = E'' g(x), \quad \text{$E'' \in \mathbb{C}$}. \label{eq:q-Heun''}
  \end{align}
This is equivalent to 
\begin{align}
& (x-q^{h''_1 + 1/2} t_1) (x-q^{h''_2 +1/2} t_2) g(x/q) +  q^{\alpha ''_1 +\alpha ''_2} (x - q^{l''_1 -1/2} t_1) (x - q^{l''_2 -1/2} t_2) g(qx)  \label{eq:q-Heun''2} \\
& -\{ (q^{\alpha ''_1} +q^{\alpha ''_2} ) x^2 + E'' x  + q^{(h''_1 +h''_2 + l''_1 +l''_2 +\alpha ''_1 +\alpha ''_2 )/2} ( q^{\beta ''/2} + q^{-\beta ''/2} ) t_1 t_2  \} g(x) = 0 . \nonumber 
\end{align}
Set 
\begin{equation}
\lambda_{1}'' = (h_{1}'' + h_{2}'' - l_{1}'' - l_{2}'' - \alpha_{1}'' - \alpha_{2}'' - \beta'' + 2)/2. \label{eq:lamabda''}
\end{equation}
The condition that the formal series 
\begin{equation*}
g(x) =  x^{\lambda ''_{1}} \sum_{n =0}^{\infty } c_{n} x^{n} 
\end{equation*}
 is a solution to equation (\ref{eq:q-Heun''2}) is equivalent to that the equality
\begin{align}
 & c_{n} t_{1}t_{2}q^{  1- n + h''_{1} + h''_{2} - \lambda ''_{1}} (1 - q^{n})(1 - q^{n - \beta '' }) \label{eq:recrelcn''} \\
 & \ - c_{n - 1} [E'' + q^{3/2- n - \lambda ''_{1} } (q^{h''_{1}}t_{1} + q^{h''_{2}}t_{2} ) + q^{n - 3/2 + \lambda ''_{1} + \alpha ''_{1} + \alpha ''_{2} }( q^{l''_{1}}t_{1} + q^{l''_{2}}t_{2} ) ] \notag \\
 & \ + c_{n - 2} q^{ 2- n - \lambda ''_{1} } (1 - q^{n - 2 + \lambda ''_{1} + \alpha ''_{1} })(1 - q^{n - 2 + \lambda ''_{1} + \alpha ''_{2} }) = 0 . \notag
  \end{align}
holds for $n = 1 ,2,3, \cdots $, where $c_{-1} =0$.
Note that the equality for $n=0$ also holds for any $c_0 \in \mathbb{C}$ by setting $c_{-2} =0$.
If $n \in \mathbb{Z} _{>0}$ and $\beta '' \not \in \{ 1,2, \dots ,n \}$, then the coefficient $c_n$ is determined uniquely by the recursive relation (\ref{eq:recrelcn''}).

We recall the polynomial-type solutions.
\begin{prop} [{cf.\ \cite[Proposition 2.2]{KST}}] \label{prop:polynom}
Let $\lambda_{1}'' = (h_{1}'' + h_{2}'' - l_{1}'' - l_{2}'' - \alpha_{1}'' - \alpha_{2}'' - \beta'' + 2)/2$, $\alpha '' \in \{  \alpha ''_1 , \alpha ''_2 \}$ and assume that $-\lambda ''_1 - \alpha '' (=: N)$ is a non-negative integer and $\beta '' \not \in \{ 1,2,\dots ,N\}$.
Set $c_{-1}(E'')=0 $,  $c_0(E'')=1 $,
\begin{align}
 & x''_n =  t_{1}t_{2}q^{  1- n + h''_{1} + h''_{2} - \lambda ''_{1}} (1 - q^{n})(1 - q^{n - \beta '' }) \label{eq:xyz''} \\
 & y''_n =  q^{3/2- n - \lambda ''_{1} } (q^{h''_{1}}t_{1} + q^{h''_{2}}t_{2} ) + q^{n - 3/2 + \lambda ''_{1} + \alpha ''_{1} + \alpha ''_{2} }( q^{l''_{1}}t_{1} + q^{l''_{2}}t_{2} )  \notag \\
 & z''_n = q^{ 2- n - \lambda ''_{1} } (1 - q^{n - 2 + \lambda ''_{1} + \alpha ''_{1} })(1 - q^{n - 2 + \lambda ''_{1} + \alpha ''_{2} }) \notag
\end{align}
and we determine the polynomials $c_n(E'')$ $(n=1,\dots ,N)$ recursively by
\begin{align}
 & c_{n} (E'') x''_n = c_{n - 1} (E'') [E'' + y''_n ] -c_{n-2} (E'') z''_n . \label{eq:cnxyz''}
\end{align}
We define the monic polynomial $c (E'')$ of degree $N+1$ by
\begin{align}
 & c (E'') = x''_1 x''_2 \cdots x''_{N} [ c_{N} (E'') \{ E'' + y''_{N+1} \} -c_{N-1} (E'') z''_{N+1} ] . \label{eq:cxyz''}
\end{align}
Assume that $E''=E''_0$ is a solution of the algebraic equation
\begin{equation}
c (E'')=0 .
\label{eq:cN+1}
\end{equation}
Then the $q$-Heun equation (\ref{eq:q-Heun''}) has a non-zero solution of the form
\begin{equation}
g(x)= x^{\lambda ''_1} \sum _{n=0}^{N } c_n (E''_0)  x^n .
\label{eq:gxxlapol}
\end{equation}
\end{prop}
\begin{proof}
We set $c_n = c_n (E''_0)$ for  $n=0,1,\dots ,N$ and $c_n =0$ otherwise.
It is enough to show equation (\ref{eq:recrelcn''}) for $n = 1 ,2, \cdots , N+2$.
If $n = 1 ,2, \cdots , N$, then equation (\ref{eq:recrelcn''}) follows from equation (\ref{eq:cnxyz''}).
We have $c (E'' _0)=0$, and it follows from equation (\ref{eq:cnxyz''}) and $c_{N+1} =0$ that equation (\ref{eq:recrelcn''}) for $n = N+1$ holds.
On equation (\ref{eq:recrelcn''}) for $n = N+2$, the coefficient of $c_N $ is equal to $0$, because $N+ \lambda ''_1 + \alpha '' _1= 0$ or $N+ \lambda ''_1 + \alpha '' _2= 0$.
Hence, equation (\ref{eq:recrelcn''}) for $n = N+2$ holds.
\end{proof}
If $N=0$ in Proposition \ref{prop:polynom}, then $c (E'') = E'' + y''_1$ and the function $x^{\lambda ''_1}$ is a solution to the $q$-Heun equation (\ref{eq:q-Heun''}) under the condition $E'' =- y''_1 $.

We have
\begin{align*}
 & c_{1}(E'') = \frac{E'' + y''_1}{x''_1} , \\
 & c_{2}(E'') = \frac{(E'' + y''_1 )  (E'' + y'' _2 )}{x''_1 x''_2}  - \frac{z''_2}{x''_2} , \notag \\
 & c_{3}(E'') = \frac{(E'' + y''_1 ) (E'' + y''_2) (E'' + y''_3) }{x''_1 x''_2 x''_3} - \frac{z''_2 (E'' + y''_3 )}{x''_2 x''_3} - \frac{z''_3 (E'' + y''_1) }{x''_1 x''_3} , \notag \\
 & c_{4}(E'') = \frac{(E'' + y''_1 ) (E'' + y''_2) (E'' + y''_3) (E'' + y''_4) }{x''_1 x''_2 x''_3 x''_4} - \frac{z''_2 (E'' + y''_3 ) (E'' + y''_4) }{x''_2 x''_3 x''_4}  \notag \\
& \qquad \qquad - \frac{z''_3 (E'' + y''_1) (E'' + y''_4)}{x''_1 x''_3 x''_4} - \frac{z''_4 (E'' + y''_1 )  (E'' + y'' _2 )}{x''_1 x''_2 x''_4}  + \frac{z''_2 z''_4}{x''_2 x''_4} ,  \notag 
\end{align*}
\begin{align*}
 & c (E'') = \left\{
\begin{array}{ll} 
E'' + y''_1 , & N=0 , \\
(E'' + y''_1 )  (E'' + y'' _2 ) - x''_1 z''_2 , & N=1 ,
\end{array}
\right. \\
& c (E'') =  (E'' + y''_1 ) (E'' + y''_2) (E'' + y''_3)  - x''_1 z''_2 (E'' + y''_3 ) - x''_2 z''_3 (E'' + y''_1) 
\end{align*}
for $N=2$, and
\begin{align*}
 c (E'') = & (E'' + y''_1 ) (E'' + y''_2) (E'' + y''_3) (E'' + y''_4) - x''_1 z''_2 (E'' + y''_3 ) (E'' + y''_4) \\
& - x''_2 z''_3 (E'' + y''_1) (E'' + y''_4) - x''_3 z''_4 (E'' + y''_1 ) (E'' + y'' _2 ) + x''_1 z''_2 x''_3 z''_4 \notag 
\end{align*}
for $N=3$.
If $N $ is a non-negative integer, then we have
\begin{align}
 & c (E'') = \sum (-1) ^k x''_{i_1} z'' _{i_1 +1 } x''_{i_2} z'' _{i_2 +1 } \cdots x''_{i_k} z'' _{i_k +1 }  \prod _{j \in \{ 1,2, \dots ,N +1 \} \setminus \atop{\{ i_1 , i_1 +1 , \dots , i_k , i_k +1  \} }} (E'' + y''_j ) , \label{eq:cE''sum}
\end{align}
where the summation is over the integers $1 \leq i_1 <i_2 < \dots <i_k \leq N$ such that $i_{l+1} -i_{l} \geq 2$ for $l =1, \dots , k-1$, which includes the case $k=0$.

\subsection{Apparent singularity} \label{sec:appsing}
We investigate the series solution
\begin{equation}
g(x) = x^{\lambda ''_{1}} \sum_{n =0}^{\infty } c_{n} x^{n}  \label{eq:sollam''}
\end{equation}
to the $q$-Heun equation (\ref{eq:q-Heun''}) in the case that $\beta '' $ is a positive integer.
Note that the value $\lambda ''$ is fixed by equation (\ref{eq:lamabda''}) and the exponents of equation (\ref{eq:q-Heun''2}) at $x=0$ are $\lambda ''_{1} $ and $\lambda ''_{1} +\beta ''$ (see \cite{TakqH} for the definition of exponents), and the value $\beta ''  $ is the difference of the exponents.
Set $\beta '' =N+1$ and assume that $N$ is a non-negative integer.
Then, the coefficients $c_n$ for $n=1,\dots ,N$ are determined by equation (\ref{eq:recrelcn''}), and equation (\ref{eq:recrelcn''}) for $n=N+1$ is written as 
\begin{align}
 & \ - c_{N} [E'' + q^{1/2- N - \lambda ''_{1} } (q^{h''_{1}}t_{1} + q^{h''_{2}}t_{2} ) + q^{N - 1/2 + \lambda ''_{1} + \alpha ''_{1} + \alpha ''_{2} }( q^{l''_{1}}t_{1} + q^{l''_{2}}t_{2} ) ] \label{eq:recrelcN+1''} \\
 & \ + c_{N - 1} q^{ 1- N - \lambda ''_{1} } (1 - q^{N - 1 + \lambda ''_{1} + \alpha ''_{1} })(1 - q^{N - 1 + \lambda ''_{1} + \alpha ''_{2} }) = 0 . \notag
  \end{align}
If the left hand side of equation (\ref{eq:recrelcN+1''}) is not equal to $0$, then a contradiction occurs and there is no solution written as equation (\ref{eq:sollam''}) with the condition $c_0 \neq 0$.
If the left hand side of equation (\ref{eq:recrelcN+1''}) is equal to $0$, then the coefficient $c_{N+1}$ can be taken any value, and the coefficients $c_n$ for $n=N+2,N+3, \cdots $ are determined recursively.
In this case, equation (\ref{eq:recrelcN+1''}) is equivalent to $c (E'') = 0$ described in Proposition \ref{prop:polynom}, and the dimension of the solutions in the form of equation (\ref{eq:sollam''}) is two.
This case corresponds to the apparency of the regular singularity $x=0$, which is analogues to the apparent singularity of the linear ordinary differential equation.

\subsection{First family of solutions} \label{sec:1sol}

We apply a single gauge transformation and a $q$-integral transformation to the polynomial-type solutions of the $q$-Heun equation.
Then, we obtain solutions of the $q$-Heun equation with other parameters, if the conditions in Theorem \ref{thm:q-trans} hold.
We consider this procedure attentively.
\begin{thm} \label{thm:polynom1gauge}
Assume that $( h_{1}' - h_{2}' -l_{1}' + l_{2}' + \alpha_{1}' - \alpha_{2}' + \beta' - 2)/2 (=N)$ is a non-negative integer.
Let $c_n(E')$ $(n=1,\dots ,N)$ and $c (E')$ be the polynomials determined by Proposition \ref{prop:polynom}, where
\begin{align*}
h''_{1} = l' _1, \;  l''_{1} = h' _1, \; h''_{2} = h' _2, \; l''_{2} = l' _2, \; \alpha ''_{1} = \alpha '_{1} , \; \alpha ''_{2} = \alpha '_{2} , \; \beta '' = \beta ' , \;  E''=E'.
\end{align*}
Let $E'_0$ be a zero of the polynomial $c (E'') $, i.e.~$c(E'_0 )=0$.
\\
(i) The functions 
 \begin{align} \label{eq:1h1h2}
& h_1 (s) = s^{(h_{1}' + h_{2}' - l_{1}' - l_{2}' - \alpha_{1}' - \alpha_{2}' - \beta' + 2)/2 } \frac{(s/(q^{l_{1}' - 1/2}t_{1}); q)_\infty}{(s/(q^{h_{1}' - 1/2}t_{1}); q)_\infty} \sum_{n = 0}^{N} c_{n} (E'_0) s^{n} ,\\
& h_2 (s)= s^{- \alpha_{2}' -N} \frac{(q^{h_{1}' + 1/2}t_{1}/s; q)_\infty}{(q^{l_{1}' + 1/2}t_{1}/s; q)_\infty} \sum_{n = 0}^{N} c_{n} (E'_0) s^{n} \notag
  \end{align}
satisfy $A^{\langle 4 \rangle} (s; h_{1}', h_{2}', l_{1}', l_{2}', \alpha_{1}', \alpha_{2}', \beta') h _i (s) = E_{0}' h_i (s)$ $(i=1,2)$.
\\
(ii) Let $\xi $ be the variable which is independent of the variable $x$ or proportional to $x$ (i.~e.~$\xi =Ax$ where $A $ is independent of $x$).
If $\beta' < 0, \, \alpha_{1}' < \alpha_{2}'$, then the Jackson integrals
    \begin{align}
      &g _1 (x) = (1 - q)x^{-\alpha_{1}} \frac{(q^{-l_{1}' + 1/2}\xi/t_{1}, q^{-h_{2}' + l_{2}' - N}\xi/x; q)_{\infty}}{(q^{-h_{1}' + 1/2}\xi/t_{1}, \xi/x; q)_{\infty}} \label{eq:1g1g2} \\
      &\quad \times \sum_{n = -\infty}^{\infty} \frac{(q^{-h_{1}' + 1/2}\xi/t_{1}, \xi/x; q)_{n}}{(q^{-l_{1}' + 1/2}\xi/t_{1}, q^{-h_{2}' + l_{2}' - N}\xi/x; q)_{n}} \sum_{k = 0}^{N} q^{(-\beta' + 1 +k )n} \xi^{-\beta' + 1 +k} c_{k} (E'_0 ) , \notag \\
      &g _2 (x) = (1 - q) x^{-\alpha_{1} - h_{2}' + l_{2}' - N} \frac{(q^{h_{1}' + 1/2}t_{1}/\xi, qx/\xi; q)_{\infty}}{(q^{l_{1}' + 1/2}t_{1}/\xi, q^{h_{2}' - l_{2}' + N + 1}x/\xi; q)_{\infty}}  \notag \\
      &\; \times \sum_{n = -\infty}^{\infty} \frac{(q^{l_{1}' + 1/2}t_{1}/\xi, q^{h_{2}' - l_{2}' + N + 1}x/\xi; q)_{n}}{(q^{h_{1}' + 1/2}t_{1}/\xi, qx/\xi; q)_{n}}  \sum_{k = 0}^{N} q^{(-\alpha_{1}' + \alpha_{2}' + 1 + N -k )n} \xi^{\alpha_{1}' - \alpha_{2}' - 1 - N + k}  c_{k} (E'_0 ) \notag
    \end{align}
converge and satisfy the $q$-Heun equation
\begin{align*}
& A^{\langle 4 \rangle} (x; h_{1}, h_{2}, l_{1}, l_{2}, \alpha_{1}, \alpha_{2}, \beta) g_i (x) = E g _i (x) 
\end{align*}
for $i=1,2$, whose parameters are given by
\begin{align}
& \alpha_{2} = \alpha_{1} - \alpha_{1}' + \alpha_{2}' + h_{2}' - l_{2}' + 1 + N , \; \beta = h_{1}' - l_{1}' + \alpha_{1}' - \alpha_{2}' - 1 - N, \label{eq:1param} \\
& l_{1} = l_{1}' , \; l_{2} = l_{2}', \;  h_{1} = h_{1}' -h_{2}' + l_{2}' - 1 - N, \; h_{2} = l_{2}' - 1 - N, \; E = q^{\alpha_{1} - \alpha_{1}'}E_{0}' . \notag
\end{align}
\end{thm}
Note that the condition $h_{2} = l_{2} - 1 - N$ is satisfied on the parameters in the theorem.
\begin{proof}
Set $\lambda_{1}'' = (h_{1}'' + h_{2}'' - l_{1}'' - l_{2}'' - \alpha_{1}'' - \alpha_{2}'' - \beta'' + 2)/2$ and $\lambda ' = (l_{1}'  - l_{2}' - h_{1}' + h_{2}' - \alpha_{1}' - \alpha_{2}' - \beta' + 2)/2$.
Then, we have $\lambda_{1}'' = \lambda ' $ and the assumption of $N$ is written as $-\lambda ' - \alpha_{2}' = - \lambda_{1}''  - \alpha_{2}'' =N$.
It follows from Proposition \ref{prop:polynom} that
  \begin{align*}
  A^{\langle 4 \rangle} (s; l_{1}', h_{2}', h_{1}', l_{2}', \alpha_{1}', \alpha_{2}', \beta') s^{\lambda '} \sum_{n = 0}^{N} c_{n}(E_{0}')s^{n} = E_{0}' s^{\lambda '} \sum_{n = 0}^{N} c_{n} (E_{0}') s^{n} .
  \end{align*}
We set $ f(s) = s^{\lambda '} \sum_{n = 0}^{N} c_{n}(E_{0}')s^{n}$ and we apply Proposition \ref{prop:GT}.
Then, we obtain (i).

We show (ii).
To apply Theorem \ref{thm:q-trans} (i) for the function $h_1(s)$ in equation (\ref{eq:1h1h2}), we look into the assumption of Theorem \ref{thm:q-trans} in this setting.
Namely, we investigate equation (\ref{eq:A4hslim}) for $h(s)=h_1(s)$.
If $s= q^{L}\xi$, then
  \begin{align*}
    &h(s) s^{-(h_{1}' + h_{2}' - l_{1}' - l_{2}' - \alpha_{1}' - \alpha_{2}' + \beta' + 2)/2} = (q^{L}\xi)^{-\beta'} \frac{(q^{L - l_{1}' + 1/2}\xi/t_{1}; q)_{\infty}}{(q^{L - h_{1}' + 1/2}\xi/t_{1}; q)_{\infty}} \sum_{n = 0}^{N} c_{n}(E_{0}')  (q^{L}\xi)^{n} .
  \end{align*}
Since
  \begin{align*}
 \frac{(q^{L - l_{1}' + 1/2}\xi/t_{1}; q)_{\infty}}{(q^{L - h_{1}' + 1/2}\xi/t_{1}; q)_{\infty}} \to 1, \; \sum_{n = 0}^{N} c_{n} (E_{0}')  (q^{L}\xi)^{n} \to c_{0} \quad \mbox{as } L \to \infty,
  \end{align*}
a sufficient condition for $C_1=0$ in equation (\ref{eq:A4hslim}) is the condition $\beta' < 0 $.

If $s= q^{K}\xi$, then
  \begin{align*}
    h(s) s^{\alpha_{1}'} 
    &= (q^{K}\xi)^{h_{1}' - l_{1}'} \frac{(q^{K - l_{1}' + 1/2}\xi/t_{1}; q)_\infty}{(q^{K - h_{1}' + 1/2}\xi/t_{1}; q)_\infty} (q^{K}\xi)^{\lambda ' + \alpha_{1}'} \sum_{n = 0}^{N} c_{n} (E_{0}')  (q^{K}\xi)^{n} .
  \end{align*}
It follows from equation (\ref{def:theta}) that
 \begin{align*}
    \frac{(q^{K - l_{1}' + 1/2}\xi/t_{1}; q)_\infty}{(q^{K - h_{1}' + 1/2}\xi/t_{1}; q)_\infty} &= \frac{\vartheta_{q} (q^{K - l_{1}' + 1/2}\xi/t_{1})}{(q^{- K + l_{1}' + 1/2}t_{1}/\xi; q)_{\infty} (q; q)_{\infty}} \frac{(q^{- K + h_{1}' + 1/2}t_{1}/\xi; q)_{\infty} (q; q)_{\infty}}{\vartheta_{q} (q^{K - h_{1}' + 1/2}\xi/t_{1})} \\
    &=\frac{\vartheta_{q} (q^{K - l_{1}' + 1/2}\xi/t_{1})}{\vartheta_{q} (q^{K - h_{1}' + 1/2}\xi/t_{1})} \frac{(q^{- K + h_{1}' + 1/2}t_{1}/\xi; q)_\infty}{(q^{- K + l_{1}' + 1/2}t_{1}/\xi; q)_\infty} .
  \end{align*}
By applying Lemma \ref{lem:theta}, we have
  \begin{align*}
 & h(s)s^{\alpha_{1}'} = (q^{K}\xi)^{h_{1}' - l_{1}'} \left(\frac{q^{- h'_{1}}}{q^{- l'_{1}}}\right)^{K} \frac{\vartheta_{q} (q^{- l_{1}' + 1/2}\xi/t_{1})}{\vartheta_{q} (q^{- h_{1}' + 1/2}\xi/t_{1})} \frac{(q^{- K + h_{1}' + 1/2}t_{1}/\xi; q)_\infty}{(q^{- K + l_{1}' + 1/2}t_{1}/\xi; q)_\infty} \\
    &\qquad \qquad \times (q^{K}\xi)^{\lambda ' + \alpha_{1}'} \sum_{n = 0}^{N} c_{n} (E_{0}') (q^{K}\xi)^{n} \\
    & \quad = \xi^{h_{1}' - l_{1}'} \frac{\vartheta_{q} (q^{- l_{1}' + 1/2}\xi/t_{1})}{\vartheta_{q} (q^{- h_{1}' + 1/2}\xi/t_{1})} \frac{(q^{- K + h_{1}' + 1/2}t_{1}/\xi; q)_\infty}{(q^{- K + l_{1}' + 1/2}t_{1}/\xi; q)_\infty} 
    (q^{K}\xi)^{\lambda ' + \alpha_{1}' +N} \sum_{n = 0}^{N} c_{n} (E_{0}') (q^{K}\xi)^{n-N} .
  \end{align*}
Since $(q^{- K + h_{1}' + 1/2}t_{1}/\xi; q)_\infty / (q^{- K + l_{1}' + 1/2}t_{1}/\xi; q)_\infty \to 1 \quad (K \to - \infty) $ and $\lambda ' + \alpha_{1}' +N = \alpha_{1}' - \alpha_{2}'  $, a sufficient condition for $C_2 =0$ in equation (\ref{eq:A4hslim}) is the condition $\alpha_{1}' < \alpha_{2}' $.
Therefore, if $\beta' < 0, \alpha_{1}' < \alpha_{2}'$, then we can apply Theorem \ref{thm:q-trans} for the function $h(s) = h_1(s)$ in equation (\ref{eq:1h1h2}).
It is shown similarly 
For the function $h(s) = h_2(s)$ in equation (\ref{eq:1h1h2}) we can show similarly that $C_1=C_2=0$ and Theorem \ref{thm:q-trans} is applicable, if $\beta' < 0, \alpha_{1}' < \alpha_{2}'$ (see \cite{Mm}).
Note that we used Lemma \ref{lem:theta} to show $C_1 =0$ under the condition $\beta' < 0 $.

We apply Theorem \ref{thm:q-trans} with the condition $\mu _0=0$.
It follows from
  \begin{align*}
    &N = -\lambda ' - \alpha_{2}' = -(l_{1}' + h_{2}' - h_{1}' - l_{2}' - \alpha_{1}' + \alpha_{2}' - \beta' + 2)/2 \\
    &\Leftrightarrow \beta' = 2N + l_{1}' + h_{2}' - h_{1}' - l_{2}' - \alpha_{1}' + \alpha_{2}' + 2
  \end{align*}
that the parameters in Theorem \ref{thm:q-trans} are determined as
  \begin{align}
    & \chi = -h_{2}' + l_{2}' - 1 - N, \quad \mu = 1 + \chi = -h_{2}' + l_{2}' - N, \label{eq:1parampf} \\
    &  \alpha_{2} = \alpha_{1} - \alpha_{1}' + \alpha_{2}' - \chi = \alpha_{1} - \alpha_{1}' + \alpha_{2}' + h_{2}' - l_{2}' + 1 + N , \notag \\
    & \beta = -\beta' - \chi = h_{1}' - l_{1}' + \alpha_{1}' - \alpha_{2}' - 1 - N, \quad l_{1} = l_{1}',  \quad  l_{2} = l_{2}', \notag \\
    & h_{1} = h_{1}' + \chi = h_{1}' -h_{2}' + l_{2}' - 1 - N, \quad h_{2} = h_{2}' + \chi = l_{2}' - 1 - N, \quad E = q^{\alpha_{1} - \alpha_{1}'}E_{0}' , \notag
  \end{align}
which include equation (\ref{eq:1param}).
In particular, we have $h_{2} = l_{2} - 1 - N$.

We rewrite the Jackson integral $g(x)$ in equation (\ref{eq:A4Jint}) by setting $h(s)=h_1(s)$.
We have
  \begin{align*}
  g(x) & = x^{-\alpha_{1}} \int_{0}^{\xi \infty} s^{-(h_{1}' + h_{2}' - l_{1}' - l_{2}' - \alpha_{1}' - \alpha_{2}' + \beta' + 2)/2} h(s) P_{\mu, 0}^{(1)} (x, s) {\rm{d}}_{q}s \\
    &= x^{-\alpha_{1}} \int_{0}^{\xi \infty} s^{-(h_{1}' + h_{2}' - l_{1}' - l_{2}' - \alpha_{1}' - \alpha_{2}' + \beta' + 2)/2} \\
    &\qquad \times s^{\lambda ' + h_{1}' - l_{1}'} \frac{(s/(q^{l_{1}' - 1/2}t_{1}); q)_\infty}{(s/(q^{h_{1}' - 1/2}t_{1}); q)_\infty} \sum_{n = 0}^{N} c_{n} (E_{0}') s^{n} \frac{(q^{\mu} s/x; q)_{\infty}}{( s/x; q)_{\infty}} {\rm{d}}_{q}s \\
    & = x^{-\alpha_{1}} \int_{0}^{\xi \infty} s^{-\beta'} \frac{(s/(q^{l_{1}' - 1/2}t_{1}); q)_\infty}{(s/(q^{h_{1}' - 1/2}t_{1}); q)_\infty} \frac{(q^{-h_{2}' + l_{2}' - N} s/x; q)_{\infty}}{(s/x; q)_{\infty}} \sum_{k = 0}^{N} c_{k} (E_{0}') s^{k} {\rm{d}}_{q}s .
  \end{align*}
Here, we used $\beta' = 2N + l_{1}' + h_{2}' - h_{1}' - l_{2}' - \alpha_{1}' + \alpha_{2}' + 2$, $\lambda ' = - N - \alpha_{2}'$ and $\mu = -h_{2}' + l_{2}' - N$.
It follows from the definition of the Jackson integral and equation (\ref{eq:GR1.2.30}) that
  \begin{align*}
    g(x) &= (1 - q) x^{-\alpha_{1}} \sum_{n = -\infty}^{\infty} q^{n}\xi (q^{n}\xi)^{-\beta'}\frac{(q^{n}\xi/(q^{l_{1}' - 1/2}t_{1}); q)_{\infty}}{(q^{n}\xi/(q^{h_{1}' - 1/2}t_{1}); q)_\infty} \\
    &\qquad \qquad \qquad \qquad \times \frac{(q^{-h_{2}' + l_{2}' - N} q^{n}\xi/x; q)_{\infty}}{(q^{n}\xi/x; q)_{\infty}} \sum_{k = 0}^{N} c_{k} (E_{0}') (q^{n}\xi)^{k} \\
 &= (1 - q)\xi^{-\beta' + 1}x^{-\alpha_{1}}\frac{(q^{-l_{1}' + 1/2}\xi/t_{1}, q^{-h_{2}' + l_{2}' - N}\xi/x; q)_{\infty}}{(q^{-h_{1}' + 1/2}\xi/t_{1}, \xi/x; q)_{\infty}} \notag \\
    &\quad \times \sum_{n = -\infty}^{\infty} \frac{(q^{-h_{1}' + 1/2}\xi/t_{1}, \xi/x; q)_{n}}{(q^{-l_{1}' + 1/2}\xi/t_{1}, q^{-h_{2}' + l_{2}' - N}\xi/x; q)_{n}} q^{(-\beta' + 1)n} \sum_{k = 0}^{N} c_{k} (E_{0}') q^{nk} \xi^{k} . \notag 
  \end{align*}
Therefore, we obtain the function $g_1 (x)$ in equation (\ref{eq:1g1g2}), and it follows from Theorem \ref{thm:q-trans} (i) that the function $g_1 (x)$ satisfies equation (\ref{eq:thmA4gE1}).
If $\beta' < 0, \alpha_{1}' < \alpha_{2}'$, then we showed that $C_1=C_2=0$ and the function $g_1 (x)$ satisfies the $q$-Heun equation.

We apply Theorem \ref{thm:q-trans} (ii) by setting $h(s)=h_2(s)$ and we rewrite the Jackson integral $g(x)$ in equation (\ref{eq:A4Jint2}). 
We have
  \begin{align*}
  &  g(x) = x^{-\alpha_{1}} \int_{0}^{\xi \infty} s^{-(h_{1}' + h_{2}' - l_{1}' - l_{2}' - \alpha_{1}' - \alpha_{2}' + 2N + l_{1}' + h_{2}' - h_{1}' - l_{2}' - \alpha_{1}' + \alpha_{2}'  + 2 + 2)/2} s^{\lambda_{1}'}  \\
    &\qquad \times \frac{(q^{h_{1}' + 1/2}t_{1}/s; q)_\infty}{(q^{l_{1}' + 1/2}t_{1}/s; q)_\infty} \sum_{n = 0}^{N} c_{n} (E_{0}') s^{n} \left(\frac{x}{s} \right)^{-h_{2}' + l_{2}' - N} \frac{(qx/s; q)_{\infty}}{(q^{-(-h_{2}' + l_{2}' - N) + 1} x/s; q)_{\infty}} {\rm{d}}_{q}s \\
    & \quad = x^{-\alpha_{1} -h_{2}' + l_{2}' - N} \int_{0}^{\xi \infty} s^{\alpha_{1}' - \alpha_{2}' - 2 - N} \frac{(q^{h_{1}' + 1/2}t_{1}/s, qx/s; q)_\infty}{(q^{l_{1}' + 1/2}t_{1}/s, q^{h_{2}' - l_{2}' + N + 1} x/s; q)_\infty} \sum_{k = 0}^{N} c_{k} (E_{0}')  s^{k} {\rm{d}}_{q}s . 
  \end{align*}
It follows from the definition of the Jackson integral and equation (\ref{eq:GR1.2.30}) that
  \begin{align*}
  &  g(x) = x^{-\alpha_{1} -h_{2}' + l_{2}' - N} (1 - q) \\
    & \quad \times \sum_{n = - \infty}^{\infty} (q^{n}\xi)^{\alpha_{1}' - \alpha_{2}' - 1 - N} \frac{(q^{h_{1}' + 1/2 - n}t_{1}/\xi, q^{1 - n}x/\xi; q)_\infty}{(q^{l_{1}' + 1/2 - n}t_{1}/\xi, q^{h_{2}' - l_{2}' + N + 1 -n}x/\xi; q)_\infty} \sum_{k = 0}^{N} c_{k} (E_{0}') (q^{n}\xi)^{k} \\
    & \; = (1 - q) x^{-\alpha_{1} -h_{2}' + l_{2}' - N} \frac{(q^{h_{1}' + 1/2}t_{1}/\xi, qx/\xi; q)_\infty}{(q^{l_{1}' + 1/2}t_{1}/\xi, q^{h_{2}' - l_{2}' + N + 1}x/\xi; q)_\infty} \\
    & \quad \times \sum_{n = - \infty}^{\infty} (q^{n}\xi)^{\alpha_{1}' - \alpha_{2}' - 1 - N} \frac{(q^{l_{1}' + 1/2}t_{1}/\xi, q^{h_{2}' - l_{2}' + N + 1}x/\xi; q)_{-n}}{(q^{h_{1}' + 1/2}t_{1}/\xi, qx/\xi; q)_{-n}}  \sum_{k = 0}^{N} c_{k} (E_{0}') q^{nk} \xi^{k} .
  \end{align*}
By replacing $-n$ to $n$, we obtain the function $g_2 (x)$ in equation (\ref{eq:1g1g2}), and it follows from Theorem \ref{thm:q-trans} (ii) that the function $g_2 (x)$ satisfies equation (\ref{eq:thmA4gE2}).
If $\beta' < 0, \alpha_{1}' < \alpha_{2}'$, then we have $C_1=C_2=0$ and the function $g_2 (x)$ satisfies the $q$-Heun equation.
\end{proof}

The summations in equation (\ref{eq:1g1g2}) are bilateral, and the functions $g_1(x)$ and $g_2 (x) $ are expressed as finite summations of the bilateral basic hypergeometric series $_2\psi _2$.
By specializing the value $\xi$, we obtain unilateral series.
We set $\xi = q^{l_{1}' + 1/2}t_{1}$ in $g_1 (x)$.
It follows from equation (\ref{eq:qshift0}) that 
   \begin{align*}
  &  g_1 (x) | _{\xi = q^{l_{1}' + 1/2}t_{1} } = (1 - q) x^{-\alpha_{1}} \frac{(q^{-l_{1}' + 1/2}\xi/t_{1}, q^{-h_{2}' + l_{2}' - N}\xi/x; q)_{\infty}}{(q^{-h_{1}' + 1/2}\xi/t_{1}, \xi/x; q)_{\infty}} \\
& \quad \times  \sum_{k = 0}^{N} \xi^{-\beta' + 1} \xi^{k} c_{k} (E'_0 ) \sum_{n = -\infty}^{\infty} \frac{(q^{-h_{1}' + 1/2}\xi/t_{1}, \xi/x; q)_{n}}{(q^{-l_{1}' + 1/2}\xi/t_{1}, q^{-h_{2}' + l_{2}' - N}\xi/x; q)_{n}} q^{(-\beta' + 1)n} q^{n k} \Big| _{\xi = q^{l_{1}' + 1/2}t_{1} } \\
& \; = (1 - q)x^{-\alpha_{1}}\frac{(q, q^{-h_{2}' + l_{2}' - N + l_{1}' + 1/2}t_{1}/x; q)_{\infty}}{(q^{-h_{1}' + l_{1}' + 1}, q^{l_{1}' + 1/2}t_{1}/x; q)_{\infty}} \\
    &\qquad \times  \sum_{k = 0}^{N} (q^{l_{1}' + 1/2}t_{1})^{-\beta' + 1 + k} c_{k} (E'_0 )
    {}_2\phi_1 \biggl( \begin{array}{@{}c@{}} q^{l_{1}' + 1/2}t_{1}/x, q^{-h_{1}' + l_{1}' + 1} \\
q^{-h_{2}' + l_{2}' + l_{1}' + 1/2 - N}t_{1}/x \end{array} ;q,q^{-\beta' + 1 + k} \biggr) . \notag
  \end{align*}
Heine's transformation formula for $_{2}\phi_{1}$ series is written as
  \begin{align}
    {}_2\phi_1 \biggl( \begin{array}{@{}c@{}} a, b \\
c \end{array} ;q,z \biggr)
    = \frac{(b, az; q)_{\infty}}{(c, z; q)_{\infty}}
    {}_2\phi_1 \biggl( \begin{array}{@{}c@{}} c/b, z \\
az \end{array} ;q,b \biggr), \quad (|z| < 1, \, |b| <1) \label{eq:Heine}
  \end{align}
(see \cite{GR}), and it follows that
   \begin{align}
   &  g_1 (x) | _{\xi = q^{l_{1}' + 1/2}t_{1} } = (1 - q)x^{-\alpha_{1}}\frac{(q; q)_{\infty}}{(q^{-h_{1}' + l_{1}' + 1}, q)_{\infty}} \sum_{k = 0}^{N} (q^{l_{1}' + 1/2}t_{1})^{-\beta' + 1 + k} c_{k} (E'_0 ) \label{eq:1g11} \\
    &\qquad \times \frac{(q^{-h_{1}' + l_{1}' - \beta' + 2 + k}; q)_{\infty}}{(q^{-\beta' + 1 + k}; q)_{\infty}}
    {}_2\phi_1 \biggl( \begin{array}{@{}c@{}} q^{-h_{2}' + l_{2}' - N}, q^{-\beta' + 1 + k} \\
q^{-h_{1}' + l_{1}' - \beta' + 2 + k} \end{array} ;q,q^{l_{1}' + 1/2}t_{1}/x \biggr) \notag
  \end{align}
under the condition $ |x| > |q^{l_{1}' + 1/2}t_{1}|$ and $\beta '<0$.
By setting $\xi = q^{h_{2}' - l_{2}' + N + 1}x$ in $g_1 (x)$, we also have
  \begin{align}
 & g_1 (x) | _{\xi = q^{h_{2}' - l_{2}' + N + 1}x } = (1 - q)x^{-\alpha_{1}}\frac{(q^{-l_{1}' + h_{2}' - l_{2}' + N + 3/2}x/t_{1}, q; q)_{\infty}}{(q^{-h_{1}' + h_{2}' - l_{2}' + N + 3/2}x/t_{1}, q^{h_{2}' - l_{2}' + N + 1}; q)_{\infty}} \label{eq:1g12} \\
 &  \times \sum_{k = 0}^{N} (q^{h_{2}' - l_{2}' + N + 1}x)^{-\beta' + 1 + k} c_{k} (E'_0 ) {}_2\phi_1 \biggl( \begin{array}{@{}c@{}} q^{-h_{1}' + h_{2}' - l_{2}' + 3/2 + N}x/t_{1}, q^{h_{2}' - l_{2}' + 1  + N}  \\
q^{h_{2}' - l_{1}' - l_{2}' + 3/2 + N}x/t_{1} \end{array} ;q,q^{-\beta' + 1 + k} \biggr) \notag \\
    &= (1 - q)x^{-\alpha_{1}}\frac{(q; q)_{\infty}}{(q^{h_{2}' - l_{2}' + N + 1}; q)_{\infty}} \sum_{k = 0}^{N} (q^{h_{2}' - l_{2}' + N + 1}x)^{-\beta' + 1 + k} c_{k} (E'_0 )\notag \\
    & \times \frac{(q^{h_{2}' - l_{2}' + 2 + N - \beta' + k}; q)_{\infty}}{(q^{-\beta' + 1 + k}; q)_{\infty}} \, {}_2\phi_1 \biggl( \begin{array}{@{}c@{}} q^{h_{1}' - l_{1}'}, q^{-\beta' + 1 + k} \\
q^{h_{2}' - l_{2}' + 2 + N -\beta' + k} \end{array} ;q,q^{-h_{1}' + h_{2}' - l_{2}' + 3/2 + N}x/t_{1} \biggr) \notag
  \end{align}
under the condition $|x| < | q^{ h_{1}' - h_{2}' + l_{2}' - N - 3/2}t_{1}|  $ and $\beta '<0$.
We set $\xi = q^{h_{1}' - 1/2}t_{1}$ in $g_2 (x)$.
If $|x| < | q^{ h_{1}' - h_{2}' + l_{2}' - N - 3/2}t_{1}|  $ and $\alpha_{1}' < \alpha_{2}' $, then
 \begin{align}
 & g_{2}(x) | _{\xi = q^{h_{1}' - 1/2}t_{1} } = (1 - q) x^{-\alpha_{1} -h_{2}' + l_{2}' - N} \frac{(q, q^{-h_{1}' + 3/2}x/t_{1}; q)_{\infty}}{(q^{-h_{1}' + l_{1}' + 1}, q^{- h_{1}' + h_{2}' - l_{2}' + N + 3/2}x/t_{1}; q)_\infty}  \\
    & \times \sum_{k = 0}^{N} (q^{h_{1}' - 1/2}t_{1})^{\alpha_{1}' - \alpha_{2}' - 1 - N + k} c_{k} (E'_0 ) {}_2\phi_1 \biggl( \begin{array}{@{}c@{}} q^{-h_{1}' + l_{1}' + 1}, q^{- h_{1}' + h_{2}' - l_{2}' + N + 3/2}x/t_{1} \\
q^{-h_{1}' + 3/2}x/t_{1} \end{array} ;q,q^{-\alpha_{1}' + \alpha_{2}' + 1 + N - k} \biggr) \notag \\
    &= (1 - q) x^{-\alpha_{1} -h_{2}' + l_{2}' - N} \frac{(q; q)_{\infty}}{(q^{-h_{1}' + l_{1}' + 1}; q)_\infty}  \sum_{k = 0}^{N} (q^{h_{1}' - 1/2}t_{1})^{\alpha_{1}' - \alpha_{2}' - 1 - N + k} c_{k} (E'_0 )  \notag \\
    & \quad \times \frac{(q^{-h_{1}' + l_{1}' -\alpha_{1}' + \alpha_{2}' + 2 + N - k}; q)_{\infty}}{(q^{-\alpha_{1}' + \alpha_{2}' + 1 + N - k}; q)_{\infty}} {}_2\phi_1 \biggl( \begin{array}{@{}c@{}} q^{- h_{2}' + l_{2}' - N}, q^{-\alpha_{1}' + \alpha_{2}' + 1 + N - k} \\
q^{-h_{1}' + l_{1}' -\alpha_{1}' + \alpha_{2}' + 2 + N - k} \end{array} ;q,q^{- h_{1}' + h_{2}' - l_{2}' + N + 3/2}x/t_{1} \biggr) . \notag
  \end{align}
We set $\xi = x$ in $g_2 (x)$.
If $|x| > | q^{l_{1}' + 1/2}t_{1}| $ and $\alpha_{1}' < \alpha_{2}' $, then
\begin{align}
  &  g_{2}(x) | _{\xi = x } = (1 - q)x^{- \alpha_{1} +\alpha_{1}' - \alpha_{2}' - 1 - h_{2}' + l_{2}' - 2N} \frac{(q^{h_{1}' + 1/2}t_{1}/x, q; q)_{\infty}}{(q^{l_{1}' + 1/2}t_{1}/x, q^{h_{2}' - l_{2}' + N + 1}; q)_{\infty}} \\
    & \qquad \qquad \times \sum_{k = 0}^{N} x^{k} c_{k} (E'_0 ) 
    {}_2\phi_1 \biggl( \begin{array}{@{}c@{}} q^{l_{1}' + 1/2}t_{1}/x, q^{h_{2}' - l_{2}' + N + 1} \\
q^{h_{1}' + 1/2}t_{1}/x \end{array} ;q,q^{-\alpha_{1}' + \alpha_{2}' + 1 + N - k} \biggr) \notag \\
   &= (1 - q)x^{ - \alpha_{1} + \alpha_{1}' - \alpha_{2}' - 1 - h_{2}' + l_{2}' - 2N} \frac{(q; q)_{\infty}}{(q^{h_{2}' - l_{2}' + N + 1}; q)_{\infty}} \notag \\
    &\quad \times \sum_{k = 0}^{N} x^{k} c_{k} (E'_0 ) \frac{(q^{h_{2}' - l_{2}' -\alpha_{1}' + \alpha_{2}' + 2 + 2N - k}; q)_{\infty}}{(q^{-\alpha_{1}' + \alpha_{2}' + 1 + N - k}; q)_{\infty}} {}_2\phi_1 \biggl( \begin{array}{@{}c@{}} q^{h_{1}' - l_{1}'}, q^{-\alpha_{1}' + \alpha_{2}' + 1 + N - k} \\
q^{h_{2}' - l_{2}' -\alpha_{1}' + \alpha_{2}' + 2 + 2N - k} \end{array} ;q,q^{l_{1}' + 1/2}t_{1}/x \biggr) . \notag
  \end{align}
These four functions are solutions to the $q$-Heun equation in the setting of Theorem \ref{thm:polynom1gauge}.

In Theorem \ref{thm:polynom1gauge}, the parameters $h_{1}' ,h_{2}', \dots ,\beta' $ are given with the condition $( h_{1}' - h_{2}' -l_{1}' + l_{2}' + \alpha_{1}' - \alpha_{2}' + \beta' - 2)/2 =N$ $(N \in {\mathbb Z}_{\geq 0})$, and the parameters $h_{1} ,h_{2}, \dots ,\beta $ are determined by equation (\ref{eq:1param}).
Then, the relation $h_{2} = l_{2} - 1 - N $ holds.
We rewrite the theorem by starting from the parameters $h_{1} ,h_{2}, \dots ,\beta $ with the condition $h_{2} = l_{2} - 1 - N $ and by adding the functions obtained by specializing the value $\xi $.
\begin{thm} \label{thm:1sols}
Assume that $h_{2} = l_{2} - 1 - N$ and $N$ is a non-negative integer.
Set  $\lambda_{1} = (h_{1} + h_{2} - l_{1} - l_{2} - \alpha_{1} - \alpha_{2} - \beta + 2)/2$, $\lambda_{2} = \lambda_{1} + \beta $,  $c_{-1}(E)=0 $,  $c_0(E)=1 $,
\begin{align*}
 & x_n = t_{1}t_{2}q^{ -n  + l_1 + l_{2} + \alpha_{2} } (1 - q^{n})(1 - q^{n -1 + \alpha _1 + \lambda _1 + \beta }) , \\
 & y_n = ( q^{1 - n  - \beta }  + q^{n - 1 -N } )q^{-\lambda _1 +h_{1} +1/2 }t_{1} + ( q^{1- n + \alpha_{2} } + q^{n -1 -N+ \alpha _1 } ) q^{l_{2} -1/2} t_{2} , \notag \\ 
 & z_n =  q^{ n -N - 2  + \alpha_{1}   } (1 - q^{N - n +1 + \alpha_{2} + \lambda _1 })(1 - q^{N -n + 2 }) \notag
\end{align*}
and we determine the polynomials $c_n(E)$ $(n=1,\dots ,N)$ recursively by
\begin{align}
 & c_{n} (E) x_n = c_{n - 1} (E) [E + y_n ] -c_{n-2} (E) z_n .  \label{eq:1cnE}
\end{align}
We define the polynomial $c (E)$ by
\begin{align}
 & c (E) = x_1 x_2 \cdots x_{N} [ c_{N} (E) \{ E + y_{N+1} \} -c_{N-1} (E) z_{N+1} ] . \label{eq:1cE}
\end{align}
Assume that $E=E_0$ is a solution of the algebraic equation $c (E)=0 $.
\\
(i) Let $\xi $ be the variable which is independent of the variable $x$ or proportional to $x$.
If $\lambda _{2} + \alpha _{1} >1 $ and  $ \lambda _{1} + \alpha _{2}  > 1$,  then the Jackson integrals
    \begin{align}
 & g _1 (x) = (1 - q) x^{-\alpha_{1}} \frac{(q^{-l_{1} + 1/2}\xi/t_{1}, q^{\lambda _{1} + \alpha _{1} }\xi/x; q)_{\infty}}{(q^{ \lambda _{1} -h_1 + \alpha _{1} - 1/2 }\xi/t_{1}, \xi/x; q)_{\infty}} \label{eq:1thmg1g2}\\
 & \qquad \times \sum_{n = -\infty}^{\infty} \frac{(q^{\lambda _{1} -h_1 + \alpha _{1} - 1/2}\xi/t_{1}, \xi/x; q)_{n}}{(q^{-l_{1} + 1/2}\xi/t_{1}, q^{\lambda _{1} + \alpha _{1} }\xi/x; q)_{n}}  \sum_{k = 0}^{N} q^{( \lambda _{1} + \alpha _{1} +\beta +k )n} \xi^{ \lambda _{1} + \alpha _{1} +\beta + k} c_{k} (E_0 ) , \notag \\
 & g _2 (x) = (1 - q) x^{ \lambda _{1} } \frac{(q^{ - \lambda _{1} + h_1 - \alpha _{1} + 3/2}t_{1}/\xi, qx/\xi; q)_{\infty}}{(q^{l_{1} + 1/2}t_{1}/\xi, q^{ - \lambda _{1} - \alpha _{1}  + 1}x/\xi; q)_{\infty}} \notag \\
 & \quad \times \sum_{n = -\infty}^{\infty} \frac{(q^{l_{1} + 1/2}t_{1}/\xi, q^{ - \lambda _{1} - \alpha _{1}  + 1}x/\xi; q)_{n}}{(q^{ - \lambda _{1} + h_1 - \alpha _{1} + 3/2}t_{1}/\xi, qx/\xi; q)_{n}} \sum_{k = 0}^{N} q^{( \lambda _{1} + \alpha _{2} + N -k)n} \xi^{ -\lambda _{1} - \alpha _{2} - N +k}  c_{k} (E_0 ) \notag
    \end{align}
converge and the functions $g _i (x)$ $(i=1,2)$ satisfy the $q$-Heun equation
\begin{align}
& A^{\langle 4 \rangle} (x; h_{1}, h_{2}, l_{1}, l_{2}, \alpha_{1}, \alpha_{2}, \beta) g _i (x) = E_0 g _i (x) .
\label{eq:1thmqH}
\end{align}
(ii) The following functions $g _i (x) $ $(i=3,4,5,6)$ are solutions to the $q$-Heun equation (\ref{eq:1thmqH}).
\begin{align*}
& g _3 (x) = x^{ \lambda _{1} } \sum_{k = 0}^{N} (q^{- \lambda _{1} + h_1 - \alpha _{1} + 1/2}t_{1})^{ k} c_{k} (E_0 )  \\
    & \quad \times \frac{(q^{ - \beta + 1 - k}; q)_{\infty}}{(q^{ \lambda _{1} + \alpha _{2} + N - k}; q)_{\infty}} {}_2\phi_1 \biggl( \begin{array}{@{}c@{}} q^{\lambda _{1} + \alpha _{1} }, q^{ \lambda _{1} + \alpha _{2} + N - k} \\
q^{ - \beta + 1 - k} \end{array} ;q,q^{ -h_1 + 1/2}x/t_{1} \biggr) , \\
 &  g _4 (x)  = x^{ \lambda _{2}} \sum_{k = 0}^{N} (q^{ - \lambda _{1} - \alpha _{1}  + 1}x)^{ k} c_{k} (E_0 ) \\
    &\quad \times \frac{(q^{\beta + k  + 1 }; q)_{\infty}}{(q^{ \lambda _{2} + \alpha _{1} + k}; q)_{\infty}} \, {}_2\phi_1 \biggl( \begin{array}{@{}c@{}} q^{\lambda _{2} + \alpha_{2} + N }, q^{ \lambda _{2} + \alpha _{1} + k} \\
q^{\beta  + k  + 1 } \end{array} ;q,q^{ -h_1  + 1/2 }x/t_{1} \biggr) , \\
&  g _5 (x)
   =x^{-\alpha_{1}}  \sum_{k = 0}^{N} (q^{l_{1} + 1/2}t_{1})^{ k} c_{k} (E_0 ) \\
    &\quad \times \frac{(q^{ \alpha_{1} - \alpha_{2} -N + k +1 }; q)_{\infty}}{(q^{ \lambda _{2} + \alpha _{1} + k}; q)_{\infty}}
    {}_2\phi_1 \biggl( \begin{array}{@{}c@{}} q^{\lambda _{1} + \alpha _{1} }, q^{ \lambda _{2} + \alpha _{1} + k} \\
q^{ \alpha_{1} - \alpha_{2} -N + k +1} \end{array} ;q,q^{l_{1} + 1/2}t_{1}/x \biggr) , \\
   & g _6 (x) = x^{ - \alpha _{2} }  \sum_{k = 0}^{N} x^{k-N} c_{k} (E_0 ) \\
    &\quad \times \frac{(q^{ - \alpha _{1} + \alpha _{2} + 1 + N - k}; q)_{\infty}}{(q^{ \lambda _{1} + \alpha _{2} + N - k}; q)_{\infty}} {}_2\phi_1 \biggl( \begin{array}{@{}c@{}} q^{\lambda _{2} + \alpha_{2} + N }, q^{ \lambda _{1} + \alpha _{2} + N - k} \\
q^{- \alpha _{1} + \alpha _{2} + 1 + N - k} \end{array} ;q,q^{l_{1} + 1/2}t_{1}/x \biggr).
  \end{align*}
\end{thm}
\begin{proof}
Set
\begin{align}
    &\chi = (h_{1} + h_{2} - l_{1} - l_{2} + \alpha_{1} - \alpha_{2} - \beta )/2  = \lambda _{1} + \alpha _{1} - 1 , \; E' = E, \;  \alpha '_{1} = \alpha _{1}, \label{eq:1'_''} \\
    & \beta ' = -\beta - \chi, \; \alpha '_{2} = \alpha_{2} + \chi, \;  l'_{1} = l_{1}, \;  l'_{2} = l_{2}, \; h'_1 = h_1 - \chi, \; h'_2 = h_2 - \chi, \notag \\
& h''_{1} = l' _1, \;  l''_{1} = h' _1, \; h''_{2} = h' _2, \; l''_{2} = l' _2, \; \alpha ''_{1} = \alpha '_{1} , \; \alpha ''_{2} = \alpha '_{2} , \; \beta '' = \beta ' , \;  E''=E'. \notag 
 \end{align}
It follows from $h_{2} = l_{2} - 1 - N$ that $-\lambda ''_1 - \alpha '' _2 = N$, where $\lambda_{1}'' := (h_{1}'' + h_{2}'' - l_{1}'' - l_{2}'' - \alpha_{1}'' - \alpha_{2}'' - \beta'' + 2)/2$.
We have $\beta '' \not \in \{ 1,2,\dots ,N\}$, because $\beta '' = - \lambda _{2} - \alpha _{1} + 1<0$. 
Hence, the assumption of Proposition \ref{prop:polynom} holds.
Let $x''_n$, $y''_n $, $z''_n $ be the values in equation (\ref{eq:xyz''}).
It follows from equation (\ref{eq:1'_''}) that $x''_n= x_n$, $y''_n= y_n$, $z''_n= z_n$, and the polynomials  $c_{n} (E'')$ and $c (E'')$ in Proposition \ref{prop:polynom} coincide with the ones defined by equations (\ref{eq:1cnE}) and (\ref{eq:1cE}).

The condition $\beta' < 0 $ and $\alpha_{1}' < \alpha_{2}'$ is equivalent to the condition $\lambda _{2} + \alpha _{1} >1 $ and  $ \lambda _{1} + \alpha _{2}  > 1$, and we apply Theorem \ref{thm:polynom1gauge}.
The relations in equations (\ref{eq:1param}) and (\ref{eq:1parampf}) follow from equation (\ref{eq:1'_''}).

The functions $g_1(x)$ and $g_2(x)$ in Theorem \ref{thm:polynom1gauge} are written as $g_1(x)$ and $g_2(x)$ in equation (\ref{eq:1thmg1g2}).
The functions $g_3(x)$, $g_4(x)$, $g_5(x)$ and $g_6(x)$ are obtained from the functions $g_{2}(x) | _{\xi = q^{h_{1}' - 1/2}t_{1} }$, $ g_1 (x) | _{\xi = q^{h_{2}' - l_{2}' + N + 1}x }$, $g_1 (x) | _{\xi = q^{l_{1}' + 1/2}t_{1} }$, $ g_{2}(x) | _{\xi = x } $ by multiplying constants appropriately.
\end{proof}
We established in Theorem \ref{thm:1sols} that there exist non-zero solutions to the $q$-Heun equation which are expressed as a finite summation of the $q$-hypergeometric functions if $l_{2} - h_{2}$ is a positive integer and the eigenvalue $E$ satisfies $c(E)=0$ where $c(E)$ is a polynomial determined in Theorem \ref{thm:1sols}.

The case $N=0$ in this subsection, which is equivalent to $ l_2 -h_2 =1$, was partly discussed in \cite{TakK}.
In this case, the polynomial $c(E)$ is written as $c(E) =E +y_1 $ where $y_1 = ( q^{- \beta  } + 1 )q^{ -\lambda _1 +h_{1} +1/2 }t_{1} + ( q^{ \alpha_ 1 } + q^{\alpha _2} ) q^{l_2 - 1/2 }t_{2} $, and the eigenvalue $E_0$ in Theorem \ref{thm:1sols} is described as $E_0= -y_1 $.
The function $g_1 (x)$ in equation (\ref{eq:1g1g2}) on the case $N=0$ was already obtained in \cite[(5.3)]{TakK} by the $q$-integral transformation of the monomial-type function, and the functions in equations (\ref{eq:1g11}) and (\ref{eq:1g12}) were also obtained in \cite[p.18]{TakK}.
Hence, Theorem \ref{thm:1sols} is a significant extension of the results in \cite[p.18]{TakK}.

\subsection{Second family of solutions} \label{sec:2sol}

We apply another gauge transformation and a $q$-integral transformation to the polynomial-type solutions of the $q$-Heun equation.
%
\begin{thm} \label{thm:polynom2gauge}
Assume that $( h_{1}' + h_{2}' -l_{1}' - l_{2}' + \alpha_{1}' - \alpha_{2}' - \beta' - 2)/2 (=N)$ is a non-negative integer.
Let $c_n(E')$ $(n=1,\dots ,N)$ and $c (E')$ be the polynomials determined by Proposition \ref{prop:polynom}, where
\begin{align*}
h''_{1} = l' _1, \; h''_{2} = l' _2, \;  l''_{1} = h' _1, \; l''_{2} = h' _2, \; \alpha ''_{1} = \alpha '_{1} , \; \alpha ''_{2} = \alpha '_{2} , \; \beta '' = -\beta ' , \;  E''=E'.
\end{align*}
Let $E'_0$ be a zero of the polynomial $c (E'') $, i.e.~$c(E'_0 )=0$.
\\
(i) The functions
 \begin{align}
& h_1 (s) = s^{(h_{1}' + h_{2}' - l_{1}' - l_{2}' - \alpha_{1}' - \alpha_{2}' + \beta' + 2)/2 } \frac{(s/(q^{l_{1}' - 1/2}t_{1}), s/(q^{l_{2}' - 1/2}t_{2}); q)_{\infty}}{(s/(q^{h_{1}' - 1/2}t_{1}), s/(q^{h_{2}' - 1/2}t_{2}); q)_{\infty}}  \sum_{n = 0}^{N} c_{n} (E'_0) s^{n} , \label{eq:2h1h2} \\
& h_2 (s) = s^{- \alpha_{2}' -N} \frac{(q^{h_{1}' + 1/2}t_{1}/s, q^{h_{2}' + 1/2}t_{2}/s; q)_{\infty}}{(q^{l_{1}' + 1/2}t_{1}/s, q^{l_{2}' + 1/2}t_{2}/s; q)_{\infty}} \sum_{n = 0}^{N} c_{n} (E'_0) s^{n} \notag
  \end{align}
satisfy $A^{\langle 4 \rangle} (s; h_{1}', h_{2}', l_{1}', l_{2}', \alpha_{1}', \alpha_{2}', \beta') h _i (s) = E_{0}' h_i (s)$ $(i=1,2)$.
\\
(ii) Let $\xi $ be the variable which is independent of the variable $x$ or proportional to $x$.
If $\alpha_{1}' < \alpha_{2}'$, then the Jackson integral
  \begin{align}
    &g _1 (x) = (1 - q) x^{-\alpha_{1}} \frac{(q^{-l_{1} + 1/2}\xi/t_{1}, q^{-l_{2} + 1/2}\xi/t_{2}, q^{\chi + 1}\xi/x; q)_{\infty}}{(q^{-h_{1} + \chi + 1/2}\xi/t_{1}, q^{-h_{2} + \chi + 1/2}\xi/t_{2}, \xi/x; q)_{\infty}} \label{eq:2g1} \\
    &\qquad \qquad \times \sum_{n = -\infty}^{\infty} \frac{(q^{-h_{1} + \chi + 1/2}\xi/t_{1}, q^{-h_{2} + \chi + 1/2}\xi/t_{2}, \xi/x; q)_{n}}{(q^{-l_{1} + 1/2}\xi/t_{1}, q^{-l_{2} + 1/2}\xi/t_{2}, q^{\chi + 1}\xi/x; q)_{n}} \sum_{k = 0}^{N} q^{(k+1)n} \xi^{k+1} c_{k}(E'_0) \notag
  \end{align}
converges and it satisfies
  \begin{align}
    &A^{\langle 4 \rangle} (x; h_{1}, h_{2}, l_{1}, l_{2}, \alpha_{1}, \alpha_{2}, \beta) g_1 (x) \label{eq:2A4g1} \\
    &\qquad = E_0 g _1 (x) - (1 - q) x^{- \alpha_{1}}q^{\alpha_{1} + h_{1} + h_{2} - \chi}(q^{-\chi - 1 - N} - 1)t_{1}t_{2} , \notag
  \end{align}
where 
\begin{align}
&\chi = -\beta' - 1 - N, \, \beta = 1 + N, \, \alpha_{2} = \alpha_{1} - \alpha_{1}' + \alpha_{2}' - \chi, \, l_{1} = l_{1}', \, l_{2} = l_{2}', \label{eq:2param} \\ 
& h_{1} = h_{1}' + \chi, \, h_{2} = h_{2}' + \chi, \, E = q^{\alpha - \alpha_{1}'}E_{0}' . \notag
\end{align}
(iii) Under the same settings as (ii), the Jackson integral
  \begin{align}
 & g _2 (x) = (1 - q)x^{-\alpha_{1} + \chi + 1} \frac{(q^{h_{1} - \chi + 1/2}t_{1}/\xi, q^{h_{2} - \chi + 1/2}t_{2}/\xi, qx/\xi; q)_{\infty}}{(q^{l_{1} + 1/2}t_{1}/\xi, q^{l_{2} + 1/2}t_{2}/\xi, q^{-\chi}x/\xi; q)_{\infty}}  \label{eq:2g2} \\
& \qquad \qquad \times \sum_{n = -\infty}^{\infty} \frac{(q^{l_{1} + 1/2}t_{1}/\xi, q^{l_{2} + 1/2}t_{2}/\xi, q^{-\chi}x/\xi; q)_{n}}{(q^{h_{1} - \chi + 1/2}t_{1}/\xi, q^{h_{2} - \chi + 1/2}t_{2}/\xi, qx/\xi; q)_{n}} \notag \\
& \qquad \qquad \qquad \times \sum_{k = 0}^{N} q^{(h_{1} + h_{2} - l_{1} - l_{2} - \chi -k )n} \xi ^{ -h_{1} - h_{2} + l_{1} + l_{2} + \chi +k } c_{k}(E'_0) \notag
  \end{align}
converges and it satisfies
  \begin{align}
    & A^{\langle 4 \rangle} (x; h_{1}, h_{2}, l_{1}, l_{2}, \alpha_{1}, \alpha_{2}, \beta) g(x) = E_0 g(x) - (1 - q)k_{1}(x), \label{eq:2A4g2} \\
    & k_{1}(x) = x^{- \alpha _1+ \chi +1} q^{\alpha_{1} + h_{1} + h_{2} - \chi } \xi^{-h_{1} - h_{2} + l_{1} + l_{2} + \chi -1} \notag \\
    & \qquad \qquad \times \frac{\vartheta _q (q^{h_{1} - \chi + 1/2}t_{1}/\xi) \vartheta _q (q^{h_{2} - \chi + 1/2}t_{2}/\xi) \vartheta_{q} \big(qx/\xi \big)}{\vartheta _q (q^{l_{1} + 1/2}t_{1}/\xi) \vartheta _q (q^{l_{2} + 1/2}t_{2}/\xi) \vartheta_{q} (q^{-\chi } x/\xi )} \big( q^{ -\chi - 1 -N} - 1 \big) t_{1} t_{2} , \notag
  \end{align}
where the parameters satisfy equation (\ref{eq:2param}).
\end{thm}
Note that $\beta = N+1 $ in the theorem.
The value $\chi $ in the theorem is also expressed as
\begin{align*}
    & \chi = (h_{1} + h_{2} - l_{1} - l_{2} + \alpha_{1} - \alpha_{2} - \beta )/2 ,
\end{align*}
which follows from equation (\ref{eq:chi}).
\begin{proof}
Set $\lambda_{1}'' = (h_{1}'' + h_{2}'' - l_{1}'' - l_{2}'' - \alpha_{1}'' - \alpha_{2}'' - \beta'' + 2)/2$ and $\lambda ' = (l_{1}' + l_{2}' - h_{1}' - h_{2}' - \alpha_{1}' - \alpha_{2}' + \beta' + 2)/2$.
Then, we have $\lambda_{1}'' = \lambda ' $ and the assumption of $N$ is written as $-\lambda ' - \alpha_{2}' = - \lambda_{1}''  - \alpha_{2}'' =N$.
It follows from Proposition \ref{prop:polynom} that
  \begin{align*}
  A^{\langle 4 \rangle} (s; l_{1}', l_{2}', h_{1}', h_{2}', \alpha_{1}', \alpha_{2}', \beta') s^{\lambda '} \sum_{n = 0}^{N} c_{n}(E_{0}')s^{n} = E_{0}' s^{\lambda '} \sum_{n = 0}^{N} c_{n} (E_{0}') s^{n} .
  \end{align*}
We obtain (i) by applying Proposition \ref{prop:GT} twice.

We show (ii).
We investigate equation (\ref{eq:A4hslim}) for $h(s)=h_1(s)$.
If $s= q^{L}\xi$, then
  \begin{align*}
&h(s) s^{-(h_{1}' + h_{2}' - l_{1}' - l_{2}' - \alpha_{1}' - \alpha_{2}' + \beta' + 2)/2} \\
&\qquad = \frac{(q^{L - l_{1}' + 1/2}\xi/t_{1}, q^{L - l_{2}' + 1/2}\xi/t_{2}; q)_\infty}{(q^{L - h_{1}' + 1/2}\xi/t_{1}, q^{L - h_{2}' + 1/2}\xi/t_{2}; q)_\infty} \sum_{n = 0}^{N}  (q^{L}\xi)^{n} c_{n}  (E_{0}') \to c_{0}  (E_{0}') (=1) 
  \end{align*}
as $L \to \infty $. 
Hence, we have $C_1 =1$ in equation (\ref{eq:A4hslim}).
If $s= q^{K}\xi$, then
  \begin{align*}
 & h(s) s^{\alpha_{1}'}  = (q^{K}\xi)^{h_{1}' + h_{2}' - l_{1}' - l_{2}'} \frac{(q^{K - l_{1}' + 1/2}\xi/t_{1}, q^{K - l_{2}' + 1/2}\xi/t_{2}; q)_\infty}{(q^{K - h_{1}' + 1/2}\xi/t_{1}, q^{K - h_{2}' + 1/2}\xi/t_{2}; q)_\infty} \\
& \qquad \qquad \times (q^{K}\xi)^{\lambda ' + \alpha_{1}' +N} \sum_{n = 0}^{N} c_{n}(q^{K}\xi)^{n-N} . 
  \end{align*}
Note that $\lambda ' + \alpha_{1}' +N = \alpha_{1}' -\alpha_{2}' $.
It is shown by using $\vartheta _q $ as the proof of  Theorem \ref{thm:polynom1gauge} that a sufficient condition for $C_2 =0$ in equation (\ref{eq:A4hslim}) is the condition $\alpha_{1}' < \alpha_{2}' $ (see \cite{Mm} for details).
Therefore, if $ \alpha_{1}' < \alpha_{2}'$, then we can apply Theorem \ref{thm:q-trans} (i) for the function $h(s) = h_1(s)$ in equation (\ref{eq:2h1h2}).
We apply Theorem \ref{thm:q-trans} with the condition $\mu _0=0$.
It follows from
  \begin{align*}
    &N = -\lambda ' - \alpha_{2}' = -(l_{1}' + l_{2}' - h_{1}' - h_{2}' - \alpha_{1}' + \alpha_{2}' + \beta' + 2)/2 \\
    &\Leftrightarrow \beta' = - 2N - l_{1}' - l_{2}' + h_{1}' + h_{2}' + \alpha_{1}' - \alpha_{2}' - 2
  \end{align*}
that the parameters in Theorem \ref{thm:q-trans} are determined as equation (\ref{eq:2param}) and $\mu = 1 + \chi $. 
In particular, we have $\beta = 1 + N$.

We rewrite the Jackson integral $g(x)$ in equation (\ref{eq:A4Jint}) by setting $h(s)=h_1(s)$.
We have
   \begin{align*}
    & g(x) = x^{-\alpha_{1}} \int_{0}^{\xi \infty} s^{-(h_{1}' + h_{2}' - l_{1}' - l_{2}' - \alpha_{1}' - \alpha_{2}' + \beta' + 2)/2} h(s) P_{\chi +1, 0}^{(1)} (x, s) {\rm{d}}_{q}s \\
    & \quad = x^{-\alpha_{1}} \int_{0}^{\xi \infty} s^{-(h_{1}' + h_{2}' - l_{1}' - l_{2}' - \alpha_{1}' - \alpha_{2}' + \beta' + 2)/2} \\
    & \qquad \times s^{\lambda ' + h_{1}' + h_{2}' - l_{1}' - l_{2}'} \frac{(s/(q^{l_{1}' - 1/2}t_{1}), s/(q^{l_{2}' - 1/2}t_{2}); q)_\infty}{(s/(q^{h_{1}' - 1/2}t_{1}), s/(q^{h_{2}' - 1/2}t_{2}); q)_\infty} \sum_{n = 0}^{N} c_{n} (E'_0) s^{n} \frac{(q^{\chi +1} s/x; q)_{\infty}}{( s/x; q)_{\infty}} {\rm{d}}_{q}s \\
    & \quad = x^{-\alpha_{1}} \int_{0}^{\xi \infty} \frac{(s/(q^{l_{1}' - 1/2}t_{1}), s/(q^{l_{2}' - 1/2}t_{2}); q)_\infty}{(s/(q^{h_{1}' - 1/2}t_{1}), s/(q^{h_{2}' - 1/2}t_{2}); q)_\infty} \frac{(q^{\chi +1 } s/x; q)_{\infty}}{(s/x; q)_{\infty}}\sum_{k = 0}^{N} c_{k} (E'_0) s^{k} {\rm{d}}_{q}s .
  \end{align*}
Here, we used $\beta' = - 2N - l_{1}' - l_{2}' + h_{1}' + h_{2}' + \alpha_{1}' - \alpha_{2}' - 2$ and $\lambda ' = - N - \alpha_{2}'$.
It follows from the definition of the Jackson integral and equation (\ref{eq:GR1.2.30}) that
  \begin{align}
    g(x) &= (1 - q) x^{-\alpha_{1}} \sum_{n = -\infty}^{\infty} q^{n}\xi \frac{(q^{n - l_{1}' + 1/2}\xi/t_{1}, q^{n - l_{2}' + 1/2}\xi/t_{2}; q)_{\infty}}{(q^{n - h_{1}' + 1/2}\xi/t_{1}, q^{n - h_{2}' + 1/2}\xi/t_{2}; q)_\infty} \label{eq:polynom-J_Int} \\
    &\qquad \times \frac{(q^{\chi + 1 + n}\xi/x; q)_{\infty}}{(q^{n}\xi/x; q)_{\infty}} \sum_{k = 0}^{N} c_{k} (E'_0) (q^{n}\xi)^{k} \notag \\
 &= (1 - q) x^{-\alpha_{1}}\frac{(q^{- l_{1} + 1/2}\xi/t_{1}, q^{- l_{2} + 1/2}\xi/t_{2}, q^{\chi + 1}\xi/x; q)_{\infty}}{(q^{- h_{1} + \chi + 1/2}\xi/t_{1}, q^{- h_{2} + \chi + 1/2}\xi/t_{2}, \xi/x; q)_\infty} \notag \\
    &\quad \times \sum_{n = -\infty}^{\infty} \frac{(q^{- h_{1} + \chi + 1/2}\xi/t_{1}, q^{- h_{2} + \chi + 1/2}\xi/t_{2}, \xi/x; q)_{n}}{(q^{- l_{1} + 1/2}\xi/t_{1}, q^{- l_{2} + 1/2}\xi/t_{2}, q^{\chi + 1}\xi/x; q)_{n}} \sum_{k = 0}^{N} q^{(k+1)n} \xi^{k+1} c_{k}  (E'_0) . \notag
  \end{align}
Therefore, we obtain the function $g_1 (x)$ in equation (\ref{eq:2g1}), and it follows from Theorem \ref{thm:q-trans} (i) that the function $g_1 (x)$ satisfies equation (\ref{eq:thmA4gE1}).
If $\alpha_{1}' < \alpha_{2}'$, then we showed that $C_1=1$ and $C_2=0$, and the function $g_1 (x)$ satisfies equation (\ref{eq:2A4g1}).

We show (iii).
We investigate equation (\ref{eq:A4hslim}) for $h(s)=h_2(s)$.
If $s= q^{L}\xi$, then it follows from equation (\ref{def:theta}) and Lemma \ref{lem:theta} that
\begin{align*}
  &h(s) s^{- (h_{1}' + h_{2}' - l_{1}' - l_{2}' - \alpha_{1}' - \alpha_{2}' + \beta' + 2)/2} \\
  &= (q^{L}\xi)^{-h_{1}' - h_{2}' + l_{1}' + l_{2}'} \frac{(q^{h_{1}' + 1/2 - L}t_{1}/\xi, q^{h_{2}' + 1/2 - L}t_{2}/\xi; q)_\infty}{(q^{l_{1}' + 1/2 - L}t_{1}/\xi, q^{l_{2}' + 1/2 - L}t_{2}/\xi; q)_\infty} \sum_{n = 0}^{N} c_{n} (E'_0) (q^{L}\xi)^{n} \notag \\
  &= \xi^{-h_{1}' - h_{2}' + l_{1}' + l_{2}'} \frac{\vartheta_{q}(q^{h_{1}' + 1/2 }t_{1}/\xi) \vartheta_{q} ( q^{h_{2}' + 1/2 }t_{2}/\xi )}{\vartheta_{q} (q^{l_{1}' + 1/2 }t_{1}/\xi) \vartheta_{q} (q^{l_{2}' + 1/2 }t_{2}/\xi ) } \notag \\
& \qquad \times \frac{(q^{L- l_{1}' + 1/2 } \xi / t_{1}, q^{L- l_{2}' + 1/2 } \xi / t_{2}; q)_\infty}{(q^{L- h_{1}' + 1/2 } \xi / t_{1}, q^{L- h_{2}' + 1/2 } \xi / t_{2}; q)_\infty} \sum_{n = 0}^{N} c_{n} (E'_0) (q^{L}\xi)^{n} \notag \\
  & \to \xi^{-h_{1}' - h_{2}' + l_{1}' + l_{2}'} \frac{\vartheta_{q}(q^{h_{1}' + 1/2 }t_{1}/\xi) \vartheta_{q} ( q^{h_{2}' + 1/2 }t_{2}/\xi )}{\vartheta_{q} (q^{l_{1}' + 1/2 }t_{1}/\xi) \vartheta_{q} (q^{l_{2}' + 1/2 }t_{2}/\xi ) } c_{0} (E'_0) \quad ( L \to +\infty ). \notag
\end{align*}
Therefore, we have 
\begin{align}
 & C_{1} = \xi^{-h_{1} - h_{2} + l_{1} + l_{2} + 2\chi} \frac{\vartheta_q (q^{h_{1} - \chi + 1/2}t_{1}/\xi) \vartheta _q (q^{h_{2} - \chi + 1/2}t_{2}/\xi)}{\vartheta _q (q^{l_{1} + 1/2}t_{1}/\xi) \vartheta _q (q^{l_{2} + 1/2}t_{2}/\xi)} . \label{eq:2C1}
  \end{align}
If $s= q^{K}\xi$, then
  \begin{align*}
  h(s)s^{\alpha_{1}'} 
  &= \frac{(q^{h_{1}' + 1/2 - K}t_{1}/\xi, q^{h_{2}' + 1/2 - K}t_{2}/\xi; q)_\infty}{(q^{l_{1}' + 1/2 - K}t_{1}/\xi, q^{l_{2}' + 1/2 - K}t_{2}/\xi; q)_\infty} (q^{K}\xi)^{\lambda ' + \alpha_{1}' +N} \sum_{n = 0}^{N} c_{n}  (E'_0) (q^{K}\xi)^{n-N} .
  \end{align*}
Note that $\lambda ' + \alpha_{1}' +N = \alpha_{1}' -\alpha_{2}' $.
A sufficient condition for $C_2 =0$ in equation (\ref{eq:A4hslim}) is the condition $\alpha_{1}' < \alpha_{2}' $.
Therefore, if $ \alpha_{1}' < \alpha_{2}'$, then we can apply Theorem \ref{thm:q-trans} (ii) for the function $h(s) = h_2(s)$ in equation (\ref{eq:2h1h2}).

We rewrite the Jackson integral $g(x)$ in equation (\ref{eq:A4Jint2}) by setting $h(s)=h_2(s)$.
We have
\begin{align*}
 & g(x) = x^{-\alpha_{1}} \int_{0}^{\xi \infty} s^{-(h_{1}' + h_{2}' - l_{1}' - l_{2}' - \alpha_{1}' - \alpha_{2}' + \beta' + 2)/2} h(s) P_{\chi +1, 0 }^{(2)} (x, s) {\rm{d}}_{q}s \\
  &= x^{-\alpha_{1} + \chi + 1}\int_{0}^{\xi \infty} s^{- h_{1}' - h_{2}' + l_{1}' + l_{2}' - \chi - 1} \\
  &\qquad \times \frac{(q^{h_{1}' + 1/2}t_{1}/s, q^{h_{2}' + 1/2}t_{2}/s, qx/s; q)_\infty}{(q^{l_{1}' + 1/2}t_{1}/s, q^{l_{2}' + 1/2}t_{2}/s, q^{-\chi} x/s; q)_\infty} \sum_{k = 0}^{N} c_{k} (E'_0) s^{k} {\rm{d}}_{q}s . 
\end{align*}
It follows from the definition of the Jackson integral and equation (\ref{eq:GR1.2.30}) that
\begin{align*}
  &g(x) = x^{-\alpha_{1} + \chi + 1} (1 - q) \sum_{n = -\infty}^{\infty} q^{n}\xi (q^{n}\xi)^{-h_{1}' - h_{2}' + l_{1}' + l_{2}' - \chi - 1} \\
  & \qquad \times \frac{(q^{h_{1}' + 1/2 - n}t_{1}/\xi, q^{h_{2}' + 1/2 - n}t_{2}/\xi, q^{1 - n}x/\xi; q)_\infty}{(q^{l_{1}' + 1/2 - n}t_{1}/\xi, q^{l_{2}' + 1/2 - n}t_{2}/\xi, q^{-\chi - n} x/\xi; q)_\infty} \sum_{k = 0}^{N} c_{k} (E'_0) (q^{n}\xi)^{k} \\
  &= (1 - q) \xi^{-h_{1}' - h_{2}' + l_{1}' + l_{2}' - \chi} x^{-\alpha_{1} + \chi + 1} \frac{(q^{h_{1}' + 1/2}t_{1}/\xi, q^{h_{2}' + 1/2}t_{2}/\xi, qx/\xi; q)_\infty}{(q^{l_{1}' + 1/2}t_{1}/\xi, q^{l_{2}' + 1/2}t_{2}/\xi, q^{-\chi} x/\xi; q)_\infty} \\
  &\times \sum_{n = -\infty}^{\infty}  \frac{(q^{l_{1}' + 1/2}t_{1}/\xi, q^{l_{2}' + 1/2}t_{2}/\xi, q^{-\chi} x/\xi; q)_{-n}}{(q^{h_{1}' + 1/2}t_{1}/\xi, q^{h_{2}' + 1/2}t_{2}/\xi, qx/\xi; q)_{-n}} q^{(-h_{1}' - h_{2}' + l_{1}' + l_{2}' - \chi)n} \sum_{k = 0}^{N} c_{k} (E'_0) q^{nk} \xi^{k} .
\end{align*}
By replacing $-n$ to $n$ and using equation (\ref{eq:2param}), we obtain the function $g_2 (x)$ in equation (\ref{eq:2g2}), and it follows from Theorem \ref{thm:q-trans} (ii) that the function $g_2 (x)$ satisfies equation (\ref{eq:thmA4gE2}).
If $\alpha_{1}' < \alpha_{2}'$, then we have $C_2=0$ and $C_1$ satisfies equation (\ref{eq:2C1}).
Therefore, the function $g_2 (x)$ satisfies equation (\ref{eq:2A4g2}).
\end{proof}

The functions $g_1(x)$ and $g_2 (x) $ in Theorem \ref{thm:polynom2gauge} are expressed as finite summations of the bilateral basic hypergeometric series $_3\psi _3$.
If the value $\xi$ is specified suitably, we obtain unilateral series.

By setting $\xi = q^{h_{1} - \chi - 1/2}t_{1}, \, \xi = q^{h_{2} - \chi - 1/2}t_{2}, \, \xi = x$ in equation (\ref{eq:2g2}) and using equation (\ref{eq:qshift0}), we introduce the following functions
  \begin{align*}
& \tilde{g}_{3}(x) = g_2 (x) | _{\xi = q^{h_{1} - \chi - 1/2}t_{1} } = (1 - q) \frac{(q, q^{- h_{1} + h_{2} + 1}t_{2}/t_{1}, q^{- h_{1} + \chi + 3/2}x/t_{1}; q)_\infty}{(q^{-h_{1} + l_{1} + \chi + 1}, q^{-h_{1} + l_{2} + \chi + 1}t_{2}/t_{1}, q^{-h_{1} + 1/2}x/t_{1}; q)_\infty} \notag \\
  &\quad \times x^{-\alpha_{1} + \chi + 1} \sum_{k = 0}^{N} (q^{h_{1} - \chi - 1/2}t_{1})^{-h_{1} - h_{2} + l_{1} + l_{2} + \chi + k} c_{k} (E_{0}') \notag \\
  &\quad \times {}_3\phi_2 \biggl( \begin{array}{@{}c@{}} q^{-h_{1} + l_{1} + \chi + 1}, q^{-h_{1} + l_{2} + \chi + 1}t_{2}/t_{1}, q^{-h_{1} + 1/2}x/t_{1} \\
     q^{- h_{1} + h_{2} + 1}t_{2}/t_{1}, q^{- h_{1} + \chi + 3/2}x/t_{1} \end{array} ;q,q^{h_{1} + h_{2} - l_{1} - l_{2} - \chi - k} \biggr), \\
& \tilde{g}_{4}(x) = g_2 (x) | _{\xi = q^{h_{2} - \chi - 1/2}t_{2} } =  (1 - q)  \frac{(q^{h_{1} - h_{2} + 1}t_{1}/t_{2}, q, q^{- h_{2} + \chi + 3/2}x/t_{2}; q)_\infty}{(q^{-h_{2} + l_{1} + \chi + 1}t_{1}/t_{2}, q^{-h_{2} + l_{2} + \chi + 1}, q^{-h_{2} + 1/2}x/t_{2}; q)_\infty} \notag \\
  &\quad \times x^{-\alpha_{1} + \chi + 1} \sum_{k = 0}^{N} (q^{h_{2} - \chi - 1/2}t_{2})^{-h_{1} - h_{2} + l_{1} + l_{2} + \chi + k} c_{k} (E_{0}') \notag \\
  &\quad \times {}_3\phi_2 \biggl( \begin{array}{@{}c@{}} q^{-h_{2} + l_{1} + \chi + 1}t_{1}/t_{2}, q^{-h_{2} + l_{2} + \chi + 1}, q^{-h_{2} + 1/2}x/t_{2} \\
     q^{h_{1} - h_{2} + 1}t_{1}/t_{2}, q^{- h_{2} + \chi + 3/2}x/t_{2} \end{array} ;q,q^{h_{1} + h_{2} - l_{1} - l_{2} - \chi - k} \biggr), \\
& \tilde{g}_{5}(x) = g_2 (x) | _{\xi =  x } =  (1 - q) x^{-h_{1} - h_{2} + l_{1} + l_{2} -\alpha_{1} + 2\chi + 1} \frac{(q^{h_{1} - \chi + 1/2}t_{1}/x, q^{h_{2} - \chi + 1/2}t_{2}/x, q; q)_{\infty}}{(q^{l_{1} + 1/2}t_{1}/x, q^{l_{2} + 1/2}t_{2}/x, q^{-\chi}; q)_{\infty}} \notag \\
  &\quad \times \sum_{k = 0}^{N} x^k c_{k}  (E_{0}') {}_3\phi_2 \biggl( \begin{array}{@{}c@{}} q^{l_{1} + 1/2}t_{1}/x, q^{l_{2} + 1/2}t_{2}/x, q^{-\chi} \\
     q^{h_{1} - \chi + 1/2}t_{1}/x, q^{h_{2} - \chi + 1/2}t_{2}/x \end{array} ;q,q^{h_{1} + h_{2} - l_{1} - l_{2} - \chi - k} \biggr).
  \end{align*}
\begin{prop} \label{prop:tg3tg4tg5}
The functions $\tilde{g}_3 (x) , \tilde{g}_{4} (x) , \tilde{g}_{5} (x) $ satisfy the $q$-Heun equation (\ref{eq:q-Heun}).
\end{prop}
\begin{proof}
It follows from Theorem \ref{thm:polynom2gauge} (iii) that the functions $\tilde{g}_3 (x) , \tilde{g}_{4} (x) , \tilde{g}_{5} (x) $ satisfies equation (\ref{eq:2A4g2}).
Since $\xi = q^{h_{1} - \chi - 1/2}t_{1}$, $\xi = q^{h_{2} - \chi - 1/2}t_{2}$ or $\xi = x$, the function $ k_{1}(x)$ in equation (\ref{eq:2A4g2}) is identically zero and we obtain the $q$-Heun equation.
\end{proof}
By setting $\xi = q^{l_{1} + 1/2}t_{1}, \, \xi = q^{l_{2} + 1/2}t_{2}, \xi = q^{-\chi}x$ in equation (\ref{eq:2g1}) and using equation (\ref{eq:qshift0}), we introduce the following functions
  \begin{align*}
    &g_{6}(x) = g_1 (x) | _{\xi = q^{l_{1} + 1/2}t_{1} } = (1 - q)x^{-\alpha_{1}}\frac{(q, q^{l_{1} - l_{2} + 1}t_{1}/t_{2}, q^{l_{1} + \chi + 3/2}t_{1}/x; q)_{\infty}}{(q^{-h_{1} + l_{1} + \chi + 1}, q^{-h_{2} + l_{1} + \chi + 1}t_{1}/t_{2}, q^{l_{1} + 1/2}t_{1}/x; q)_{\infty}} \\
    &\qquad \times \sum_{k = 0}^{N} (q^{l_{1} + 1/2}t_{1})^{k+1} c_k (E_{0}') {}_3\phi_2 \biggl( \begin{array}{@{}c@{}} q^{-h_{1} + l_{1} + \chi + 1}, q^{-h_{2} + l_{1} + \chi + 1}t_{1}/t_{2}, q^{l_{1} + 1/2}t_{1}/x \\
q^{l_{1} - l_{2} + 1}t_{1}/t_{2}, q^{l_{1} + \chi + 3/2}t_{1}/x \end{array} ;q,q^{k+1} \biggr), \\
    &g_{7}(x)  = g_1 (x) | _{\xi = q^{l_{2} + 1/2}t_{2} } =  (1 - q)x^{-\alpha_{1}}\frac{(q^{- l_{1} + l_{2} + 1}t_{2}/t_{1}, q, q^{l_{2} + \chi + 3/2}t_{2}/x; q)_{\infty}}{(q^{-h_{1} + l_{2} + \chi + 1}t_{2}/t_{1}, q^{-h_{2} + l_{2} + \chi + 1}, q^{l_{2} + 1/2}t_{2}/x; q)_{\infty}} \\
    &\qquad \times \sum_{k = 0}^{N} (q^{l_{2} + 1/2}t_{2})^{k+1}  c_k (E_{0}') {}_3\phi_2 \biggl( \begin{array}{@{}c@{}} q^{-h_{1} + l_{2} + \chi + 1}t_{2}/t_{1}, q^{-h_{2} + l_{2} + \chi + 1}, q^{l_{2} + 1/2}t_{2}/x \\
q^{- l_{1} + l_{2} + 1}t_{2}/t_{1}, q^{l_{2} + \chi + 3/2}t_{2}/x \end{array} ;q,q^{k+1} \biggr), \\
    &g_{8}(x)  = g_1 (x) | _{\xi = q^{-\chi}x } =  (1 - q)x^{-\alpha_{1}}\frac{(q^{-l_{1} - \chi + 1/2}x/t_{1}, q^{-l_{2} - \chi + 1/2}x/t_{2}, q; q)_{\infty}}{(q^{-h_{1} + 1/2}x/t_{1}, q^{-h_{2} + 1/2}x/t_{2}, q^{-\chi}; q)_{\infty}} \\
    &\qquad \times \sum_{k = 0}^{N} (q^{- \chi}x)^{k+1}  c_k (E_{0}') {}_3\phi_2 \biggl( \begin{array}{@{}c@{}} q^{-h_{1} + 1/2}x/t_{1}, q^{-h_{2} + 1/2}x/t_{2}, q^{-\chi} \\
q^{-l_{1} - \chi + 1/2}x/t_{1}, q^{-l_{2} - \chi + 1/2}x/t_{2} \end{array} ;q,q^{k+1} \biggr).
  \end{align*}
The functions $g_6 (x) , g_7 (x) , g_8 (x) $ satisfy equation (\ref{eq:2A4g1}), and we have the following proposition.
\begin{prop}
Let $i, \, j \in \{6, 7, 8\}$.
The function $g_i (x) - g_{j} (x) $ satisfies the $q$-Heun equation (\ref{eq:q-Heun}).
\end{prop}

In Theorem \ref{thm:polynom2gauge}, the parameter $h_{1}' ,h_{2}', \dots ,\beta' $ are given with the condition $( h_{1}' + h_{2}' -l_{1}' - l_{2}' + \alpha_{1}' - \alpha_{2}' - \beta' - 2)/2 =N$ $(N \in {\mathbb Z}_{\geq 0})$, and the parameter $h_{1} ,h_{2}, \dots ,\beta $ are determined by equation (\ref{eq:2param}).
Then, the relation $\beta = N +1 $ holds.
We rewrite the theorem by starting from the parameters $h_{1} ,h_{2}, \dots ,\beta $ with the condition $\beta = N +1  $ and by adding the functions obtained by specializing the value $\xi $.
\begin{thm} \label{thm:2sols}
Set $\lambda_{1} = (h_{1} + h_{2} - l_{1} - l_{2} - \alpha_{1} - \alpha_{2} - \beta + 2)/2$.
Let $N$ be a non-negative integer, and assume that $\beta = N +1 $ and $-\lambda_{1} -\alpha _1 \not \in \{ 0,1,\dots ,N-1 \}$.
Set $c_{-1}(E)=0 $,  $c_0(E)=1 $,
\begin{align}
 & x_n =t_{1}t_{2} q^{ n - N +1 + h_{1} + h_{2} - \alpha _1 - 2 \lambda _1 } (1 - q^{-n})(1 - q^{N-n + \lambda _1 + \alpha _1 }) , \label{eq:2xyz}  \\
 & y_n =  q^{N -n + 1/2 + \lambda _1 + \alpha_{1} + \alpha_{2} } (q^{l_{1}}t_{1} + q^{l_{2}}t_{2}) + q^{n- N-1/2  - \lambda _1} (q^{h_{1}}t_{1} + q^{h_{2}}t_{2})  , \notag \\
 & z_n = q^{ n - N-2 + \alpha_{1}  } (1 - q^{N -n +1 +\lambda _1 + \alpha_{2}  })(1 - q^{N- n + 2}), \notag
\end{align}
and we determine the polynomials $c_n(E)$ $(n=1,\dots ,N)$ recursively by
\begin{align}
 & c_{n} (E) x_n = c_{n - 1} (E) [E + y_n ] -c_{n-2} (E) z_n . \label{eq:2cnE}
\end{align}
We define the polynomial $c (E)$ by
\begin{align}
 & c (E) = x_1 x_2 \cdots x_{N} [ c_{N} (E) \{ E + y_{N+1} \} -c_{N-1} (E) z_{N+1} ] . \label{eq:2cE}
\end{align}
Assume that $E=E_0$ is a solution of the algebraic equation $c (E)=0$.
\\
(i) Let $\xi $ be the variable which is independent of the variable $x$ or proportional to $x$.
If  $ \lambda _{1} + \alpha _{2}  > 1$, then the Jackson integral
  \begin{align*}
    &g _1 (x) = (1 - q) x^{-\alpha_{1}} \frac{(q^{-l_{1} + 1/2}\xi/t_{1}, q^{-l_{2} + 1/2}\xi/t_{2}, q^{ \lambda _{1} + \alpha _{1} }\xi/x; q)_{\infty}}{(q^{ \lambda _{1} -h_{1} + \alpha _{1} - 1/2}\xi/t_{1}, q^{ \lambda _{1} -h_{2}+ \alpha _{1} - 1/2}\xi/t_{2}, \xi/x; q)_{\infty}} \\
    &\qquad \times \sum_{n = -\infty}^{\infty} \frac{(q^{ \lambda _{1} -h_{1} + \alpha _{1} - 1/2}\xi/t_{1}, q^{ \lambda _{1} -h_{2} + \alpha _{1} - 1/2}\xi/t_{2}, \xi/x; q)_{n}}{(q^{-l_{1} + 1/2}\xi/t_{1}, q^{-l_{2} + 1/2}\xi/t_{2}, q^{ \lambda _{1} + \alpha _{1} }\xi/x; q)_{n}} \sum_{k = 0}^{N} q^{( k+1)n} \xi^{k+1} c_{k}(E_0) \notag
  \end{align*}
converges and it satisfies
  \begin{align}
    &A^{\langle 4 \rangle} (x; h_{1}, h_{2}, l_{1}, l_{2}, \alpha_{1}, \alpha_{2}, \beta) g (x) \label{eq:2qHnonh1} \\
    &\qquad = E_0 g (x) - (1 - q) x^{- \alpha_{1}}q^{ - \lambda _{1} + h_{1} + h_{2} + 1 }(q^{- \lambda _{1} - \alpha _{1} - N} - 1)t_{1}t_{2} . \notag
  \end{align}
(ii) If  $ \lambda _{1} + \alpha _{2}  > 1$, then the Jackson integral
  \begin{align*}
    &g _2 (x) = (1 - q) x^{ \lambda _{1} } \frac{(q^{-  \lambda _{1} + h_{1} - \alpha _{1} + 3/2}t_{1}/\xi, q^{ - \lambda _{1} +h_{2} - \alpha _{1} + 3/2}t_{2}/\xi, qx/\xi; q)_{\infty}}{(q^{l_{1} + 1/2}t_{1}/\xi, q^{l_{2} + 1/2}t_{2}/\xi, q^{- \lambda _{1} - \alpha _{1} + 1}x/\xi; q)_{\infty}} \\
    & \qquad \qquad \times \sum_{n = -\infty}^{\infty} \frac{(q^{l_{1} + 1/2}t_{1}/\xi, q^{l_{2} + 1/2}t_{2}/\xi, q^{- \lambda _{1} - \alpha _{1} + 1 }x/\xi; q)_{n}}{(q^{- \lambda _{1} + h_{1} - \alpha _{1} + 3/2}t_{1}/\xi, q^{- \lambda _{1} + h_{2} - \alpha _{1} + 3/2}t_{2}/\xi, qx/\xi; q)_{n}} \\
    & \qquad \qquad \qquad \qquad \times \sum_{k = 0}^{N} q^{( \lambda _{1}  + \alpha_{2}+ N -k)n} \xi^{- \lambda _{1}  -\alpha_{2} - N + k} c_{k}(E_0) \notag
  \end{align*}
converges and it satisfies
  \begin{align*}
    & A^{\langle 4 \rangle} (x; h_{1}, h_{2}, l_{1}, l_{2}, \alpha_{1}, \alpha_{2}, \beta) g(x) = E_0 g(x) - (1 - q) x^{ \lambda _{1} } q^{ - \lambda _{1} + h_{1} + h_{2} + 1} \xi^{- \lambda _{1} -\alpha_{2}  - N  -1}  \\
    & \quad \times \frac{\vartheta _q (q^{- \lambda _{1} + h_{1} - \alpha _{1} + 3/2}t_{1}/\xi) \vartheta _q (q^{- \lambda _{1} + h_{2} - \alpha _{1} + 3/2}t_{2}/\xi) \vartheta_{q} \big(qx/\xi \big)}{\vartheta _q (q^{l_{1} + 1/2}t_{1}/\xi) \vartheta _q (q^{l_{2} + 1/2}t_{2}/\xi) \vartheta_{q} (q^{- \lambda _{1} - \alpha _{1} + 1 } x/\xi )} \big( q^{ - \lambda _{1} - \alpha _{1} -N} - 1 \big) t_{1} t_{2} . \notag
  \end{align*}
(iii) Set
\begin{align*}
& g_{3}(x) = x^{ \lambda _{1} } \frac{( q^{ \lambda _{1} - h_{1} + \alpha _{1} + 1/2}x/t_{1}; q)_\infty}{( q^{-h_{1} + 1/2}x/t_{1}; q)_\infty}  \sum_{k = 0}^{N} (q^{- \lambda _{1} + h_{1}  - \alpha _{1} + 1/2}t_{1})^k c_{k} (E_{0})  \notag \\
  &\quad \times {}_3\phi_2 \biggl( \begin{array}{@{}c@{}} q^{ \lambda _{1} -h_{1} + l_{1} + \alpha _{1} }, q^{ \lambda _{1} -h_{1} + l_{2} + \alpha _{1} }t_{2}/t_{1}, q^{-h_{1} + 1/2}x/t_{1} \\
     q^{- h_{1} + h_{2} + 1}t_{2}/t_{1}, q^{ \lambda _{1} - h_{1} + \alpha _{1} + 1/2}x/t_{1} \end{array} ;q,q^{ \lambda _{1}  + \alpha_{2} +N  - k} \biggr), \\
  &  g_{4}(x) = x^{ \lambda _{1} } \frac{( q^{ \lambda _{1} - h_{2} + \alpha _{1} + 1/2}x/t_{2}; q)_\infty}{( q^{-h_{2} + 1/2}x/t_{2}; q)_\infty}  \sum_{k = 0}^{N}  (q^{- \lambda _{1} +h_{2} - \alpha _{1} + 1/2}t_{2})^k c_{k} (E_{0}) \notag \\
  &\quad \times {}_3\phi_2 \biggl( \begin{array}{@{}c@{}} q^{ \lambda _{1} -h_{2} + l_{1} + \alpha _{1} }t_{1}/t_{2}, q^{ \lambda _{1} -h_{2} + l_{2} + \alpha _{1} }, q^{-h_{2} + 1/2}x/t_{2} \\
     q^{h_{1} - h_{2} + 1}t_{1}/t_{2}, q^{ \lambda _{1} - h_{2} + \alpha _{1} + 1/2}x/t_{2} \end{array} ;q,q^{ \lambda _{1}  + \alpha_{2} +N - k} \biggr), \\
 & g_{5}(x) = x^{ - \alpha_{2} } \frac{(q^{- \lambda _{1} + h_{1} - \alpha _{1} + 3/2}t_{1}/x, q^{- \lambda _{1} + h_{2} - \alpha _{1} + 3/2}t_{2}/x ; q)_{\infty}}{(q^{l_{1} + 1/2}t_{1}/x, q^{l_{2} + 1/2}t_{2}/x ; q)_{\infty}} \sum_{k = 0}^{N} x^{k-N} c_{k}  (E_{0}) \notag \\
  &\quad \times  {}_3\phi_2 \biggl( \begin{array}{@{}c@{}} q^{l_{1} + 1/2}t_{1}/x, q^{l_{2} + 1/2}t_{2}/x, q^{- \lambda _{1} - \alpha _{1} + 1} \\
     q^{- \lambda _{1} + h_{1} - \alpha _{1} + 3/2}t_{1}/x, q^{- \lambda _{1} + h_{2} - \alpha _{1} + 3/2}t_{2}/x \end{array} ;q,q^{ \lambda _{1}  + \alpha_{2} +N - k} \biggr).
\end{align*}
The functions $g_3(x)$,  $g_4(x)$ and $g_5(x)$ satisfy the $q$-Heun equation
  \begin{align*}
    & A^{\langle 4 \rangle} (x; h_{1}, h_{2}, l_{1}, l_{2}, \alpha_{1}, \alpha_{2}, \beta) g(x) = E_0 g(x) . 
  \end{align*}
(iv)
Set 
  \begin{align*}
    &g_{6}(x) = (1 - q)x^{-\alpha_{1}}\frac{(q, q^{l_{1} - l_{2} + 1}t_{1}/t_{2}, q^{ \lambda _{1} + l_{1} +\alpha _{1} + 1/2}t_{1}/x; q)_{\infty}}{(q^{ \lambda _{1} -h_{1} + l_{1} + \alpha _{1} }, q^{ \lambda _{1} -h_{2} + l_{1} + \alpha _{1} }t_{1}/t_{2}, q^{l_{1} + 1/2}t_{1}/x; q)_{\infty}} \\
    &\quad \times \sum_{k = 0}^{N} (q^{l_{1} + 1/2}t_{1})^{k+1} c_k (E_{0}) {}_3\phi_2 \biggl( \begin{array}{@{}c@{}} q^{ \lambda _{1} -h_{1} + l_{1} + \alpha _{1} }, q^{ \lambda _{1} -h_{2} + l_{1} + \alpha _{1} }t_{1}/t_{2}, q^{l_{1} + 1/2}t_{1}/x \\
q^{l_{1} - l_{2} + 1}t_{1}/t_{2}, q^{\lambda _{1} + l_{1} + \alpha _{1} + 1/2}t_{1}/x \end{array} ;q,q^{k+1} \biggr), \\
    &g_{7}(x) = (1 - q)x^{-\alpha_{1}}\frac{(q,q^{- l_{1} + l_{2} + 1}t_{2}/t_{1}, q^{\lambda _{1} + l_{2} + \alpha _{1} + 1/2}t_{2}/x; q)_{\infty}}{(q^{ \lambda _{1} -h_{2} + l_{2} + \alpha _{1} }, q^{ \lambda _{1} -h_{1} + l_{2} + \alpha _{1} }t_{2}/t_{1},  q^{l_{2} + 1/2}t_{2}/x; q)_{\infty}} \\
    &\quad \times \sum_{k = 0}^{N} (q^{l_{2} + 1/2}t_{2})^{k+1}  c_k (E_{0}) {}_3\phi_2 \biggl( \begin{array}{@{}c@{}} q^{ \lambda _{1} -h_{2} + l_{2} + \alpha _{1} }, q^{ \lambda _{1} -h_{1} + l_{2}  + \alpha _{1} }t_{2}/t_{1},  q^{l_{2} + 1/2}t_{2}/x \\
q^{- l_{1} + l_{2} + 1}t_{2}/t_{1}, q^{\lambda _{1} + l_{2} + \alpha _{1} + 1/2}t_{2}/x \end{array} ;q,q^{k+1} \biggr), \\
    &g_{8}(x) = (1 - q)x^{-\alpha_{1}}\frac{(q, q^{- \lambda _{1} -l_{1} - \alpha _{1} + 3/2}x/t_{1}, q^{- \lambda _{1} -l_{2}  - \alpha _{1} + 3/2}x/t_{2} ; q)_{\infty}}{( q^{- \lambda _{1} - \alpha _{1} + 1}, q^{-h_{1} + 1/2}x/t_{1}, q^{-h_{2} + 1/2}x/t_{2}; q)_{\infty}} \\
    &\quad \times \sum_{k = 0}^{N} (q^{- \lambda _{1} - \alpha _{1} + 1}x)^{k+1}  c_k (E_{0}) {}_3\phi_2 \biggl( \begin{array}{@{}c@{}} q^{- \lambda _{1} - \alpha _{1} + 1}, q^{-h_{1} + 1/2}x/t_{1}, q^{-h_{2} + 1/2}x/t_{2} \\
q^{- \lambda _{1} -l_{1} - \alpha _{1} + 3/2}x/t_{1}, q^{- \lambda _{1} -l_{2} - \alpha _{1} + 3/2}x/t_{2} \end{array} ;q,q^{k+1} \biggr).
  \end{align*}
The functions $g_{i} (x)- g_{j} (x)$ $(i, \, j \in \{6, 7, 8\})$ satisfy the $q$-Heun equation.
\\
(v) The functions $g_6(x)$,  $g_7(x)$ and $g_8(x)$ satisfy equation (\ref{eq:2qHnonh1}).
 \end{thm}
\begin{proof}
Set
\begin{align}
    &\chi = (h_{1} + h_{2} - l_{1} - l_{2} + \alpha_{1} - \alpha_{2} - \beta )/2  = \lambda _{1} + \alpha _{1} - 1 , \quad \alpha '_{1} = \alpha _{1} \label{eq:2'_''} \\
    &E' = E, \quad \beta ' = -\beta - \chi, \quad \alpha '_{2} = \alpha_{2} + \chi, \quad  l'_{i} = l_{i}, \quad h'_{i} = h_{i} - \chi, \quad i = 1, 2, \notag \\
& h''_{1} = l' _1, \; h''_{2} = l' _2, \;  l''_{1} = h' _1, \; l''_{2} = h' _2, \; \alpha ''_{1} = \alpha '_{1} , \; \alpha ''_{2} = \alpha '_{2} , \; \beta '' = - \beta ' , \;  E''=E'. \notag 
 \end{align}
It follows from $\beta = N +1$ that $-\lambda ''_1 - \alpha '' _2 = N $.
Since $\beta '' = \beta + \chi = N + \lambda _{1} + \alpha _{1}$, we have $\beta '' \not \in \{ 1,2,\dots ,N\}$ by the assumption.
Hence, the assumption of Proposition \ref{prop:polynom} holds.
Let $x''_n$, $y''_n $, $z''_n $ be the values in equation (\ref{eq:xyz''}).
It follows from equation (\ref{eq:2'_''}) that $x''_n= x_n$, $y''_n= y_n$, $z''_n= z_n$, and the polynomials  $c_{n} (E'')$ and $c (E'')$ in Proposition \ref{prop:polynom} coincide with the ones defined by equations (\ref{eq:2cnE}) and (\ref{eq:2cE}).
The condition $\alpha_{1}' < \alpha_{2}'$ follows from $ \lambda _{1} + \alpha _{2}  > 1$.
Hence, we obtain the theorem by Theorem \ref{thm:polynom2gauge} and the subsequent propositions.
Note that we multiplied constants appropriately to the functions $\tilde{g}_3(x)$,  $\tilde{g}_4(x)$ and $\tilde{g}_5(x)$ in Proposition \ref{prop:tg3tg4tg5} to have the functions $g_3(x)$,  $g_4(x)$ and $g_5(x)$ in (iii).
\end{proof}

\begin{prop} 
Assume that $\beta = N +1 $ and $N$ is a non-negative integer.
The condition that the singularity $x=0$ of the $q$-Heun equation is apparent is equivalent to the condition that the eigenvalue $E$ satisfies the algebraic equation $c(E)=0$ in Theorem \ref{thm:2sols}.
\end{prop}
\begin{proof}
We recall the definition that the singularity $x=0$ is apparent in the setting of the $q$-Heun equation.
Set $\lambda_{1} = (h_{1} + h_{2} - l_{1} - l_{2} - \alpha_{1} - \alpha_{2} - \beta + 2)/2$.
If the formal series $ s^{\lambda _{1}} \sum_{n =0}^{\infty } d_{n} x^{n} $ is a solution to the $q$-Heun equation, then we have
\begin{align}
 & d_{n} t_{1}t_{2}q^{  1- n + h_{1} + h_{2} - \lambda _{1}} (1 - q^{n})(1 - q^{n - \beta }) \label{eq:recreldn} \\
 & \ - d_{n - 1} [E + q^{3/2- n - \lambda _{1} } (q^{h_{1}}t_{1} + q^{h_{2}}t_{2} ) + q^{n - 3/2 + \lambda _{1} + \alpha _{1} + \alpha _{2} }( q^{l_{1}}t_{1} + q^{l_{2}}t_{2} ) ] \notag \\
 & \ + d_{n - 2} q^{ 2- n - \lambda _{1} } (1 - q^{n - 2 + \lambda _{1} + \alpha _{1} })(1 - q^{n - 2 + \lambda _{1} + \alpha _{2} }) = 0 . \notag
  \end{align}
for $n = 1 ,2,3, \cdots $, where $d_{-1} =0$.
Set $d_0=1$ and determine $d_n$ for $n=1,2,\dots ,N$ recursively by equation (\ref{eq:recreldn}).
On equation (\ref{eq:recreldn}) for $n=N+1$, the coefficint of $d_{N+1}$ vanishes, because $\beta = N +1 $.
Then, the apparency of the singularity $x=0$ is that the equation (\ref{eq:recreldn}) for $n=N+1$ holds.

Set $d_{-1}(E)=0 $,  $d_0(E)=1 $,
\begin{align*}
 & x^* _n = t_{1}t_{2}q^{  1- n + h_{1} + h_{2} - \lambda _{1}} (1 - q^{n})(1 - q^{n - \beta }) , \\
 & y^* _n = q^{3/2- n - \lambda _{1} } (q^{h_{1}}t_{1} + q^{h_{2}}t_{2} ) + q^{n - 3/2 + \lambda _{1} + \alpha _{1} + \alpha _{2} }( q^{l_{1}}t_{1} + q^{l_{2}}t_{2} ) , \notag \\
 & z^* _n = q^{ 2- n - \lambda _{1} } (1 - q^{n - 2 + \lambda _{1} + \alpha _{1} })(1 - q^{n - 2 + \lambda _{1} + \alpha _{2} }) , \notag
\end{align*}
and we determine the polynomials $d_n(E)$ $(n=1,\dots ,N)$ recursively by
\begin{align}
 & d_{n} (E) x^* _n = d_{n - 1} (E) [E + y^* _n ] -d_{n-2} (E) z^* _n . \label{eq:dnE}
\end{align}
We define the polynomial $d (E)$ by
\begin{align}
 & d (E) = x^* _1 x^* _2 \cdots x^* _{N} [ d_{N} (E) \{ E + y^* _{N+1} \} - d_{N-1} (E) z^* _{N+1} ] . \label{eq:dE}
\end{align}
Then, the singularity $x=0$ is apparent, if and only if the eigenvalue $E$ satisfies $d(E)=0$.

It is enough to show that $c(E) = d(E)$.
Let $x_n , y_n, z_n$ be the values defined in equation (\ref{eq:2xyz}).
Then, we have
\begin{equation*}
y^* _{N+2-n} =y_n , \; x^* _{N+1-n} z^* _{N+2-n} = x_n z_{n+1} .
\end{equation*}
It follows from equation (\ref{eq:cE''sum}) and the corresponding equation for $d(E)$ that
\begin{align*}
 & d (E) = \sum (-1) ^k x^* _{j_1} z^* _{j_1 +1 } x^* _{j_2} z^* _{j_2 +1 } \cdots x^* _{j_k} z^* _{j_k +1 } \! \! \! \! \! \! \! \! \! \!  \prod _{j' \in \{ 1,2, \dots ,N +1 \} \setminus \atop{\{ j_1 , j_1 +1 , \dots , j_k , j_k +1  \} }} \! \! \! \! \! \! \! \! \! \! (E + y^* _{j'} ) , \\
& = \sum (-1) ^k x^* _{N+1 - i_k} z^* _{N +2 - i_k } \cdots x^* _{N+1 -i_2} z^* _{N+2 - i_2 } x^* _{N+1 -i_1} z^* _{N+2 -i_1 } \! \! \! \! \! \! \! \! \! \! \prod _{i' \in \{ 1,2, \dots ,N +1 \} \setminus \atop{\{ i_1 , i_1 +1 , \dots , i_k , i_k +1  \} }}  \! \! \! \! \! \! \! \! \! \!  (E + y^* _{N+2-i'} ) \notag \\
 &  = \sum (-1) ^k x _{i_1} z _{i_1 +1 } x _{i_2} z _{i_2 +1 } \cdots x _{i_k} z _{i_k +1 } \! \! \! \! \! \! \! \! \! \!  \prod _{i' \in \{ 1,2, \dots ,N +1 \} \setminus \atop{\{ i_1 , i_1 +1 , \dots , i_k , i_k +1  \} }} \! \! \! \! \! \! \! \! \! \! (E + y _{i'} ) = c(E) . \notag 
\end{align*}
\end{proof}

\begin{cor} 
Assume that $\beta = N +1 $ and $N$ is a non-negative integer.
If the singularity $x=0$ of the $q$-Heun equation is apparent and $-\lambda_{1} -\alpha _1 \not \in \{ 0,1,\dots ,N-1 \}$, then there exist non-zero solutions to the $q$-Heun equation which are expressed as a finite summation of the $q$-hypergeometric functions.
The explicit forms of the solutions are described in Theorem \ref{thm:2sols}.
\end{cor}

We give comments to the case $N=0$ in this subsection, which is equivalent to $\beta = 1$.
In this case, the polynomial $c(E)$ is written as $c(E) =E +y_1 $ where $y_1 = y_1^* = q^{1/2 - \lambda _{1} } (q^{h_{1}}t_{1} + q^{h_{2}}t_{2} ) + q^{- 1/2 + \lambda _{1} + \alpha _{1} + \alpha _{2} }( q^{l_{1}}t_{1} + q^{l_{2}}t_{2} ) $, and the eigenvalue $E_0$ in Theorem \ref{thm:2sols} is described as $E_0= -y_1 $.
The $q$-Heun equation $ A^{\langle 4 \rangle} (x; h_{1}, h_{2}, l_{1}, l_{2}, \alpha_{1}, \alpha_{2}, \beta) g (x) = E_0 g (x)  $ in this setting is the variant of the $q$-hypergeometric equation of degree two (see equation (\ref{eq:varqhgdeg2})).
The functions $h_1 (x)$, $h_2(x)$ and $h_3(x)$ in equation (\ref{eq:solvarHG2}) are equal to the functions $g_3(x)$, $g_4(x)$ and $g_5 (x)$ in the case $N=0$.
By restricting Theorem \ref{thm:2sols} to the case $N=0$, the results in \cite[\S 4.1.1]{AT}, which were technically based on the $q$-middle conovlution, are recovered.

\section{Concluding remarks} \label{sec:rem}

In this paper, we obtained two families of the special solutions of the $q$-Heun equation which are written as finite summations of the $q$-hypergeometric functions by using the $q$-integral transformation.
One is described in Theorem \ref{thm:2sols}, where the parameters satisfy the condition that $\beta $ is a positive integer and the singularity $x=0$ of the $q$-difference equation is apparent.
By replacing the variable $x$ to $1/x$, we have similar solutions to the $q$-Heun equation of which the parameters satisfy $\alpha _1 - \alpha _2 \in {\mathbb Z}$ and the singularity $x=\infty $ is apparent.
Another family of the special solutions is described in Theorem \ref{thm:1sols}, where the parameters satisfy the condition that $ l_{2} - h_{2} $ is a positive integer and the accessory parameter $E$ satisfies a certain relation.
Although the relation for $E$ is described explicitly, direct interpretation from the structure of the $q$-difference equation is not obtained.
It might be related to some analogue of apparent singularities of differential equations other than the points $x=0, \infty $, and we hope to develop a theory in this direction.

We give comments on the $q$-difference equations which may have similar structures of solutions.
The $q$-Heun equation is related to the $q$-Painlev\'e VI, which was denoted by $q$-$P(D^{(1)}_5)$ from the symmetry, through the $2\times 2$ linear $q$-difference equations (see \cite{TakR}). 
The variants of the $q$-Heun equation were introduced in \cite{TakqH} and they are also related to the $q$-Painlev\'e equations $q$-$P(E^{(1)}_6)$ and $q$-$P(E^{(1)}_7)$.
Kernel functions for the variants of the $q$-Heun equation were introduced in \cite{TakK}, and $q$-integral transformations were partly investigated there.
It is expected to look into solutions of these equations.
Inversely, degenerations of the $q$-Heun equation were introduced by Sato and the second author \cite{ST}, and the studies of their solutions would have been scarcely dealt with.

\section*{Acknowledgements}
The second author was supported by JSPS KAKENHI Grant Number JP22K03368.

\end{document}